\PassOptionsToPackage{table}{xcolor}
\documentclass[review,hidelinks,onefignum,onetabnum]{siamart220329}

\usepackage{etoolbox}
\makeatletter
\let\cref@override@label@type\@undefined
\def\cref@override@label@type#1\@nil{}
\makeatother

\usepackage{lipsum}
\usepackage{amsfonts}
\usepackage{amssymb}
\usepackage{amsmath}
\usepackage{bm}
\usepackage{graphicx}
\usepackage{epstopdf}
\usepackage{algorithmic}

\usepackage{subfigure}
\usepackage{multirow}
\usepackage{color}
\usepackage{graphicx}
\usepackage{subcaption}

\usepackage{booktabs}
\usepackage{multirow}
\usepackage{diagbox}
\usepackage{arydshln} 

\usepackage{makecell} 

\usepackage{lineno}
\nolinenumbers

\usepackage{tikz}
\usetikzlibrary{calc, matrix, positioning,patterns}

\newsiamremark{remark}{Remark}
\newsiamremark{example}{Example}

\title{High-order complete flux schemes for convection-diffusion equations on arbitrary subdivisions}

\author{Wenyu Lei\footnotemark[1]\thanks{School of Mathematical Sciences, University of Electronic Science and Technology of China, Chengdu, Sichuan 611731, China 	(\email{wenyu.lei@uestc.edu.cn}, \email{xul@uestc.edu.cn}, \email{pyang@uestc.edu.cn})  }
	\and Liwei Xu\footnotemark[1]
	\and Peng Yang\footnotemark[1]
	 } 

\headers{High-order complete flux schemes on arbitrary subdivisions}{W. Lei,  L. Xu and P. Yang}

\begin{document}
	
	\maketitle
	
	\begin{abstract}
		We develop a novel complete flux finite volume method for convection-diffusion equations  on arbitrary subdivisions in two and three dimensions. Unlike standard finite volume discretizations, where the numerical flux is directly approximated from the flux definition, we derive the exact normal flux across each control volume edge/face from the underlying PDE. This exact flux splits naturally into a homogeneous part (the classical Scharfetter--Gummel flux) and an inhomogeneous part based on a Green's function that incorporates the tangential flux and the source term. The resulting formulation is exactly equivalent to the continuous equation and, once the discrete space is chosen, yields high-order schemes without using correction or stabilization strategies.   From this framework, we develop concrete numerical schemes on arbitrary grids
		using  Lagrange finite element spaces and B-spline spaces, together with their
		companion dual meshes (control volume partitions).
		Numerical experiments in two and three dimensions confirm the  optimal convergence and positivity preservation of the proposed schemes.
	
%
	\end{abstract}
	
	
	
	\section{Introduction}
	Steady-state convection-diffusion equations arise in numerous physical and engineering applications, including semiconductor device simulation~\cite{Selberherr1984}, groundwater contaminant transport~\cite{ZhengBennett2002}, heat and mass transfer in fluid flows~\cite{Patankar1980}, and atmospheric pollutant dispersion~\cite{SeinfeldPandis2016}. 
	When convection dominates over diffusion, the problem becomes singularly perturbed and its numerical approximation is notoriously challenging: standard Galerkin methods produce spurious oscillations unless the mesh is excessively refined~\cite{BrooksHughes1982}, while simple upwinding introduces excessive numerical diffusion~\cite{Patankar1980}.
	
	Finite volume methods are attractive for these problems due to their local conservation property and their natural ability to handle discontinuous coefficients and complex geometries~\cite{Aavatsmark2002,Eymard2000,Hermeline2000,LeVeque2002,LiR2000}. 
	A particularly influential class of methods for convection-dominated problems originates from the work of Scharfetter and Gummel~\cite{Scharfetter1969} in semiconductor physics. 
	The Scharfetter--Gummel (SG) scheme solves a local one-dimensional boundary value problem on the interval between two adjacent grid points, assuming constant coefficients and neglecting the source term. The resulting numerical flux, expressed in terms of the Bernoulli function, resolves sharp layers without spurious oscillations even on relatively coarse meshes.  Building on the same exponential fitting technique, related methods have been developed in various numerical frameworks, including finite element and finite volume discretizations~\cite{RoosStynesTobiska2008}, and recently discontinuous Galerkin methods~\cite{LeiW2024}.
	A more systematic generalization of this idea  is the complete flux (CF) scheme  proposed by Thiart~\cite{Thiart1990}  and  further investigated by Ten Thije Boonkkamp and co-workers~\cite{Liu2013,ten2005complete,ten2011}, who derived the CF formulation by retaining the source term contribution in the Scharfetter--Gummel flux.  In multiple spatial dimensions, the tangential flux divergence naturally enters the local ODE as part of the effective source term, leading to a generalized representation of the numerical flux~\cite{Cheng2021,ten2011}. 
	A uniform second-order convergent CF scheme on unstructured one-dimensional grids was established in~\cite{Farrell2017}, which also outlines an extension to unstructured triangular meshes; numerical experiments indicate that second-order accuracy on such meshes is attained  on locally orthogonal grids with one direction aligned along the advection vector.
	
	Despite these developments, the extension of the complete flux methodology to high-order finite volume schemes on arbitrary multi-dimensional grids remains relatively unexplored.
	On the other hand, higher-order finite volume discretizations in multiple dimensions have traditionally been pursued within several related frameworks, including the finite volume element method (FVEM)~\cite{Cai1991,ChenZ2012}, the control volume finite element method (CV-FEM)~\cite{Forsyth1991}, the box method~\cite{BankRose1987}, the vertex-centered finite volume schemes~\cite{ZhangZ2015}, etc. 
	For linear and bi-linear spaces, these approaches are well understood and second-order accuracy is obtained under mild mesh conditions~\cite{XuZou2009}. 
	For quadratic and higher even-degree polynomial spaces, however, a well-documented difficulty arises: the standard FVEM (or CV-FEM) does not achieve optimal $L^2$ convergence on arbitrary dual meshes~\cite{lv2012optimal,Wang2016}. 
	On quadrilateral meshes, optimal accuracy can be recovered only when the control volumes are constructed using the Gauss--Legendre points~\cite{LinY2015,ZhangZ2015}. 
	A similar issue has been observed in the context of isogeometric analysis with B-spline spaces~\cite{KamberG2022,KamberG2020}. For B-spline spaces, employing the Greville abscissae as control volume centers is currently the only systematic way to construct a dual partition. However, numerical evidence indicates that this choice yields suboptimal convergence rates for quadratic and higher even-degree splines, and no strategy analogous to the Gauss--Legendre points for Lagrange finite elements is currently known for B-spline spaces. The underlying mechanism for this suboptimal behavior is still not fully understood. The question of achieving optimal high-order accuracy through suitably positioned control volume boundaries remains open. 
	
	In this paper, we construct a family of complete flux finite volume schemes based on an exact representation of the normal flux derived from the underlying PDE. 
	The starting point is the observation that, for a given control volume edge, the convection-diffusion equation reduces to a one-dimensional ODE along the normal direction, parametrized by the tangential coordinate. Solving this ODE with a Green's function yields an exact expression for the normal flux that decomposes into a homogeneous part (the classical SG flux) and an inhomogeneous part involving integrals of the source term and the tangential flux divergence. By inserting this flux representation into the control volume balance,  we obtain an exact flux relation on each control volume. 
	Discrete schemes are then generated by restricting this exact relation to a chosen finite element space and evaluating the resulting integrals either analytically or by quadrature.
	
	Within this framework, we develop the following concrete schemes  using  Lagrange finite element spaces and B-spline spaces, together with their
	companion dual meshes (control volume partitions):
	\begin{itemize}
		\item A second-order scheme on general triangular (or rectangular) meshes using piecewise linear (or bi-linear) finite element  spaces,  with standard barycentric dual partitions. 
		\item A third-order scheme on rectangular meshes using the bi-quadratic $C^0$ Lagrange space, with a family of dual partitions parametrized by a scalar $p^*$.
		\item A third-order scheme on rectangular meshes using the bi-quadratic $C^1$ B-spline space, with control volumes centered at the Greville abscissae.
		\item An extension to three-dimensional problems on  hexahedral or tetrahedral meshes.
	\end{itemize}
  The proposed formulation offers two noticeable advantages. 
   First, the upwinding strength in the numerical flux is governed by two parameters that play the role of designated‌ normal distances from the edge to the adjacent control volume centers, yet they are not required to equal the actual geometric distances. They can be chosen freely, and their difference  is not fixed by the mesh size (it may be a constant multiple thereof, allowing some freedom of choice; see the discussion in Remark~\ref{remark,1} and the numerical experiments in Subsection~\ref{subsec:2D_convergence}).
	Second, for quadratic spaces, the complete flux scheme achieves optimal  convergence independently of the placement of control volume boundaries, in contrast to the standard FVEM (or CV-FEM) which requires a carefully designed dual mesh. On rectangular meshes with bi-quadratic $C^0$ Lagrange elements, the scheme attains optimal accuracy for any admissible dual mesh; for bi-quadratic $C^1$ B-spline spaces, optimal convergence is obtained at the Greville abscissae without additional tuning.
	
	The remainder of the paper is organized as follows. 
	Section~\ref{sec:2} introduces the model problem, the control volume construction, and the local coordinate system. 
	Section~\ref{sec,3} derives the exact flux relations, including the Green's function representation and the control volume balance. 
	Section~\ref{sec:4} presents the construction of concrete complete flux schemes:  second-order schemes on general triangular and rectangular meshes, third-order schemes on rectangular meshes using both $C^0$ Lagrange and $C^1$ B-spline spaces, and the extension to three dimensions. 
	Numerical experiments are reported in Section~\ref{sec:numerical}, and conclusions are given in Section~\ref{sec:conclusion}.

	\section{Preliminaries}\label{sec:2}
	\subsection{Model problem}
	
	We consider the steady-state convection-diffusion equation  on a bounded domain $\Omega \subset \mathbb{R}^d$ $(d=2,3)$:
	\begin{equation}
		-\nabla \cdot( \alpha \nabla  u - \bm{\beta}(\bm{x}) u )= f \quad \text{in } \Omega,
		\label{eq:model}
	\end{equation}
	where $\alpha>0$ is the constant diffusion coefficient, $\bm{\beta}(\bm{x})=(\beta_1(\bm{x}),\ldots,\beta_d(\bm{x}))^{T}$ is the convection velocity, and $f(\bm{x})$ is a given source term. Equation \eqref{eq:model} is supplemented by appropriate boundary conditions on $\partial \Omega$, which may be of Dirichlet, Neumann, or Robin type. 	
	In this work, we are particularly interested in the convection-dominated regime, i.e.\ the case where the diffusion constant $\alpha$ may be very small compared to the velocity $|\bm{\beta}|$. 
	
	In what follows, we focus on the two-dimensional case ($d=2$) in our description and analysis, with the understanding that the extension to the three-dimensional case ($d=3$) is conceptually straightforward, albeit notationally more involved. Although the two cases are largely similar, we point out that there are some minor differences, and we shall discuss the three-dimensional case separately in Subsection~\ref{subsec:extension_3d}.
	
	\subsection{Control volumes in two dimensions}\label{subs,1.2}
	Assume that  the domain  $\Omega\subset \mathbb{R}^2$ is a polygonal domain. Let  $\Omega$ be partitioned into a set of non‑overlapping polygonal elements, e.g., triangles, quadrilaterals, or general polygons.  Denote the partition by $\mathcal{T}_h$, where  $h$ is the maximum diameter of all elements. 
	
	\begin{figure}[h!]
		\centering
		\subfigure[Rectangular mesh]{
			\begin{minipage}[t]{.46\textwidth}
				\centering
				\includegraphics[width=140pt]{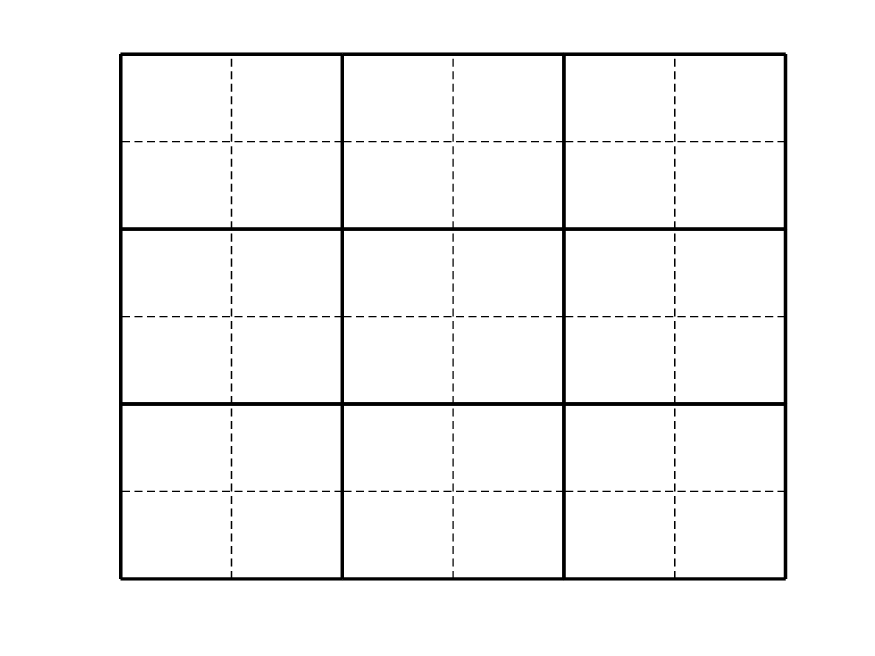}
			\end{minipage}
		}
		\subfigure[Triangular mesh]{
			\begin{minipage}[t]{.46\textwidth}
				\centering
				\includegraphics[width=140pt]{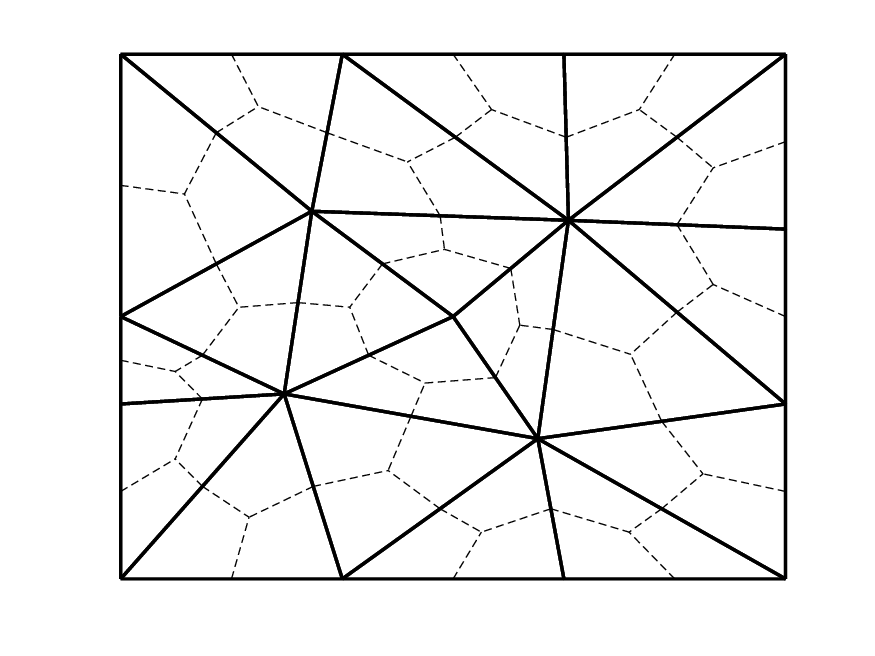}
				\label{fig,0,1}
			\end{minipage}
		}
		\caption{Dual partitions (control volume grids, dashed lines) for rectangular and triangular  primary meshes of a square domain $\Omega$.}
		\label{fig,0}
	\end{figure}
	For the control volume grid, we consider a different polygonal partition of the domain $\Omega$, which is constructed based on the primary mesh $\mathcal{T}_h$ and is therefore referred to as the dual partition of $\mathcal{T}_h$. Suppose that each dual element (control volume) $V_P$ of the dual partition is associated with a point $P$ (often taken as the centroid of $V_P$); the dual partition is then denoted by %
	\begin{eqnarray*}
		\mathcal{T}^*_h= \{V_P\}_{P\in \mathcal{V}}, 
	\end{eqnarray*} 
	where $\mathcal{V}$ is the set of such points. In a vertex-centered finite volume method, the control volume centers coincide with the computational nodes of the primary mesh. Figure~\ref{fig,0} illustrates the dual partitions (control volume grids) for triangular and rectangular primary meshes of a square domain $\Omega$: the solid lines indicate  the primary mesh edges which describe $\mathcal{T}_h$, while the dashed lines represent control volume boundaries which describe $\mathcal{T}^*_h$. The dual mesh shown for the triangular case is constructed using the barycenters and edge midpoints of the triangles.
	
	Additional dual grid constructions introduced in connection with the higher-order solution spaces will be  discussed later,  see Subsection~\ref{subsec:higher_order_rect}. In those cases, the control volumes will occupy more complex spatial positions. For instance, $P\in \mathcal{V}$ may not be a vertex of the primary mesh, or more precisely, the control volume may no longer be centered around a primary mesh vertex. In all cases, the number of control volumes matches the dimension of the discrete solution space, i.e.\ the number of unknowns.

	\subsection{Local coordinate system in two dimensions}\label{subs,1.3}
	\begin{figure}[h!]
		\centering
		\begin{minipage}[t]{.9\textwidth}
			\centering
			\includegraphics[width=180pt]{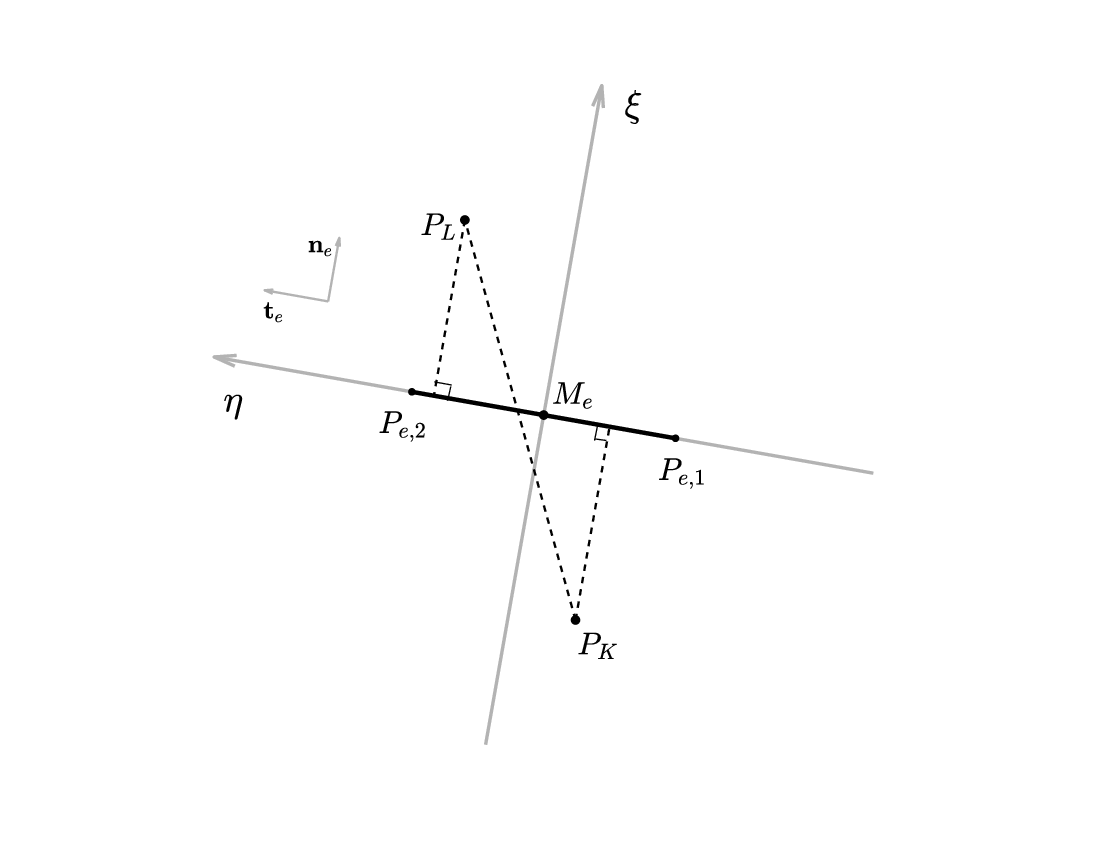}
		\end{minipage}
		\caption{Local coordinate system on a control volume edge.}
		\label{fig,1}
	\end{figure}
	Consider a control volume edge $e$, which is a finite line segment with endpoints $P_{e,1}$ and $P_{e,2}$, e.g., see Figure~\ref{fig,0}, the dashed line segment connecting the center of a primary element to an edge midpoint. We shall derive the exact flux across this edge.
	To this end, we assign an orientation (from $P_{e,1}$ to $P_{e,2}$), and we denote by $\bm{t}_e$ its unit tangential vector and by $\bm{n}_e$ its unit normal vector, chosen such that $(\bm{n}_e,\bm{t}_e)$ forms a right‑handed orthonormal system. Let $M_e$ be the midpoint of $e$. On the two sides of $e$ lie two points $P_K$ and $P_L$, which are the centers of the adjacent control volumes; in the barycentric dual mesh of Figure~\ref{fig,0,1}, they are two vertices of the primary triangle. We introduce a local Cartesian coordinate system $(\xi,\eta)$ as follows (see Figure~\ref{fig,1}):
	\begin{itemize}
		\item The origin is placed at the midpoint $M_e$.
		\item The $\eta$‑axis is aligned with the tangential direction $\bm{t}_e$, i.e.\ $\eta$ increases in the direction of $\bm{t}_e$.
		\item The $\xi$‑axis is aligned with the normal direction $\bm{n}_e$, i.e.\ $\xi$ increases in the direction of $\bm{n}_e$.
	\end{itemize}
	In this system, the edge $e$ lies on the line $\xi=0$, with $\eta$ ranging from the $\eta$‑coordinates of its endpoints. The points $P_K$ and $P_L$ have local coordinates $(\xi_K,\eta_K)$ and $(\xi_L,\eta_L)$, respectively. Without loss of generality, we assume $P_K$ on the negative side ($\xi_K\le0$) and $P_L$ on the positive side ($\xi_L\ge0$). By construction,  $\xi_K\le0\le\xi_L$ with $\xi_L-\xi_K>0$.  Note that the line segment $\overline{P_KP_L}$ is not required to be perpendicular to $e$.
	\section{Exact flux relations in two dimensions}\label{sec,3}
	
	Using local coordinates, we derive exact expressions for the flux across a control volume edge and for its contribution to the control volume balance. We adopt the local coordinate system $(\xi,\eta)$ and the associated notation introduced in Subsection~\ref{subs,1.3}. The derivation proceeds in two steps: first, we analyze the local one‑dimensional problem along the normal direction on a single edge, obtaining an exact representation of the normal flux in terms of the traces of the solution and of the source term. Second, we insert this flux representation into the integral balance over  control volumes. The resulting exact flux relation is the foundation for our numerical schemes.
	\subsection{Local equation near the edge $e$}
	
	Define the flux 
	\begin{equation}\label{flux,j}
		\bm{j} = \alpha\nabla u - \bm{\beta}u.
	\end{equation}
	Its normal and tangential components with respect to $e$ (see Figure~\ref{fig,1}) are
	\begin{eqnarray*}
		j_{\bm{n}}:=\bm{j}\cdot \bm{n}_e = \alpha \frac{\partial u}{\partial\xi}-\beta_{\bm{n}}u,\qquad j_{\bm{t}}:=\bm{j}\cdot \bm{t}_e =\alpha \frac{\partial u}{\partial\eta}-\beta_{\bm{t}}u,
	\end{eqnarray*}
	where $\beta_{\bm{n}} = \bm{\beta}\cdot\bm{n}_e$ and $\beta_{\bm{t}} = \bm{\beta}\cdot\bm{t}_e$.  So \eqref{eq:model} can be rewritten in the local coordinates as
	\begin{equation}\label{eq:local eq}
		-\nabla\cdot\bm{j}=-\frac{\partial j_{\bm{n}} }{\partial\xi} - \frac{\partial j_{\bm{t}} }{\partial\eta} = f.
	\end{equation}
	Substituting the expression for $j_{\bm{n}}$ yields
	\begin{equation}\label{eq:local}
		-\alpha \frac{\partial^2 u}{\partial\xi^2} + \frac{\partial  \big(\beta_{\bm{n},\eta}(\xi) u\big)}{\partial \xi}= f + \frac{\partial j_{\bm{t}}}{\partial\eta} =: S_\eta(\xi). 
	\end{equation}
	Here, $ \beta_{\bm{n},\eta}(\xi)$ simply denotes $\beta_{\bm{n}}(\xi,\eta)$.
	For a fixed $\eta$, \eqref{eq:local} is an ordinary differential equation in $\xi$ on the interval $(\xi_K,\xi_L)$ with Dirichlet boundary conditions
	\begin{eqnarray}\label{eq,boundary condition}
		u(\xi_K,\eta) = u_{K,\eta}, \qquad u(\xi_L,\eta) = u_{L,\eta},
	\end{eqnarray}
	where $u_{K,\eta}$ and $u_{L,\eta}$ denote the traces of $u$ along the lines $\xi=\xi_K$ and $\xi=\xi_L$, respectively. The solution of \eqref{eq:local} can be split into a homogeneous part $u^h$ and an inhomogeneous part $u^i$, i.e.\ 
	\begin{eqnarray}\label{eq:u=uh+ui}
		u = u^h + u^i. 
	\end{eqnarray}
	To this end, denote
	\begin{eqnarray}\label{eq,w}
		w(\xi) = e^{-\int_{\xi_K}^{\xi} \frac{\beta_{\bm{n},\eta}(t)}{\alpha}\; \mathrm{d}t}.
	\end{eqnarray}
	Then, the homogeneous equation 
	\begin{equation*}
		-\alpha \frac{\partial^2 u^h}{\partial\xi^2} + \frac{\partial\big(\beta_{\bm{n},\eta}(\xi) u^h\big)}{\partial \xi}
		= -\frac{\partial}{\partial\xi}\Big(\alpha \frac{\partial u^h}{\partial \xi} - \beta_{\bm{n},\eta}(\xi) u^h\Big) = 0,
	\end{equation*}
	implies the homogeneous normal flux
	\begin{equation}\label{eq: homogeneous normal flux}
	\alpha \frac{\partial u^h}{\partial \xi} - \beta_{\bm{n},\eta}(\xi) u^h = \alpha w^{-1}(\xi)\frac{\partial \big(w(\xi)u^h\big)}{\partial \xi},
  	\end{equation}
  	is constant along $\xi$.
We obtain the general solution
	\begin{eqnarray}\label{eq:homogeneous equation}
		u^h = \frac{C_1}{w(\xi)}  +\frac{ C_2}{w(\xi)}\int_{\xi_K}^{\xi} w(t)\;\mathrm{d}t.
	\end{eqnarray}
Applying the boundary condition \eqref{eq,boundary condition} to solve $C_1$ and $C_2$ gives 
	\begin{eqnarray}\label{eq:homogeneous equation,1}
		u^h = \frac{u_{K,\eta}}{w(\xi)} + \big(u_{L,\eta}w(\xi_L)-u_{K,\eta}\big) \frac{\int_{\xi_K}^{\xi} w(t)\,\mathrm{d}t}{w(\xi)\int_{\xi_K}^{\xi_L} w(t)\,\mathrm{d}t}.
	\end{eqnarray}
	 On the other hand, the inhomogeneous solution of \eqref{eq:local}  is given by
	\begin{equation} \label{eq:inhomogeneous equation}
		u^i = \int_{\xi_K}^{\xi_L} G(\xi,s) S_\eta(s)\; \mathrm{d}s,
	\end{equation}
	where $G(\xi,s)$  is the Green's function  satisfying
	\begin{equation*}
	-\alpha \frac{\partial^2 G(\xi,s)}{\partial\xi^2} + \frac{\partial\big(\beta_{\bm{n},\eta}(\xi) G(\xi,s)\big)}{\partial \xi} =\delta(\xi-s),
	\end{equation*}
	and
	\begin{equation*}
		G(\xi_K,s) = 0,\qquad G(\xi_L,s) = 0.
	\end{equation*}
	The closed formula of $G$ can be derived from two independent solutions \eqref{eq:homogeneous equation} of the homogeneous equation that satisfy the left and right boundary conditions respectively:
	\begin{equation}\label{eq,y}
	y_1(\xi) = \frac{1}{w(\xi)}\int_{\xi_K}^{\xi} w(t)\,\mathrm{d}t,\qquad
	y_2(\xi) = \frac{1}{w(\xi)}\int_{\xi}^{\xi_L} w(t)\,\mathrm{d}t.
	\end{equation}
	Their Wronskian is
	\begin{eqnarray}\label{eq,W}
		W(\xi) = y_1(\xi)y'_2(\xi)  - y_2(\xi)y'_1(\xi)  = - \frac{1}{w(\xi)}\int_{\xi_K}^{\xi_L} w(t)\;\mathrm{d}t \,\,\,(\not= 0).
	\end{eqnarray}
	The standard  Green's function construction \cite{CoddingtonLevinson1955}  gives
	\begin{eqnarray*}
		G(\xi,s) = \left\{
		\begin{aligned}
			&\frac{1}{-\alpha W(s)}y_1(\xi)y_2(s),\quad \xi_K\leq\xi\leq s,\\
			&\frac{1}{-\alpha  W(s)}y_1(s)y_2(\xi), \quad  s\leq\xi\leq\xi_L.
		\end{aligned}
		\right.
	\end{eqnarray*}
	Substituting \eqref{eq,y} and \eqref{eq,W}, we obtain the explicit expression
	\begin{eqnarray}\label{eq:G(xi,s)}
		G(\xi,s) = \left\{
		\begin{aligned}
			&\frac{\big(\int_{\xi_K}^{\xi} w(t)\,\mathrm{d}t\big)\big(\int_{s}^{\xi_L} w(t)\,\mathrm{d}t\big)}
			{\alpha\, w(\xi)\int_{\xi_K}^{\xi_L} w(t)\,\mathrm{d}t},\quad \xi_K\leq\xi\leq s,\\[6pt]
			&\frac{\big(\int_{\xi_K}^{s} w(t)\,\mathrm{d}t\big)\big(\int_{\xi}^{\xi_L} w(t)\,\mathrm{d}t\big)}
			{\alpha\, w(\xi)\int_{\xi_K}^{\xi_L} w(t)\,\mathrm{d}t},\quad  s\leq\xi\leq\xi_L.
		\end{aligned}
		\right.
	\end{eqnarray}
	\subsection{Control volume integration}
 Now we are ready to construct the exact normal flux on each control volume edge~$e$.
		Using the local coordinate system on $e$ (see Figure~\ref{fig,1}),
		we integrate the normal flux~$j_{\bm{n}}=\alpha \partial u/\partial\xi-\beta_{\bm{n},\eta}u$ along~$e$.
		Let the local coordinates of the two endpoints~$P_{e,1}$ and~$P_{e,2}$
		of edge~$e$ be~$(0,\eta_{e,1})$ and~$(0,\eta_{e,2})$, respectively.
		From~\eqref{eq:u=uh+ui}, we obtain
	\begin{align*}
		J_e:&=\int_{\eta_{e,1}}^{\eta_{e,2}} j_{\bm{n}}\big|_{\xi=0}\; \mathrm{d}\eta \\
		&= \int_{\eta_{e,1}}^{\eta_{e,2}}  (\alpha\frac{\partial u^h}{\partial \xi}-\beta_{\bm{n},\eta} u^h)\big|_{\xi=0} \;\mathrm{d}\eta + \int_{\eta_{e,1}}^{\eta_{e,2}}  (\alpha\frac{\partial u^i}{\partial \xi}-\beta_{\bm{n},\eta} u^i)\big|_{\xi=0} \;\mathrm{d}\eta \\
		&=:J_e^h+J_e^i.
	\end{align*}
	We evaluate the homogeneous part $J_e^h$ using \eqref{eq:homogeneous equation,1}. In view of \eqref{eq: homogeneous normal flux}, we obtain that
	\begin{eqnarray}\label{eq:J,h}
		J_e^h =  \int_{\eta_{e,1}}^{\eta_{e,2}} \alpha\, \frac{u_{L,\eta}w(\xi_L) - u_{K,\eta}}{\int_{\xi_K}^{\xi_L} w(t)\,\mathrm{d}t} \;\mathrm{d}\eta.
	\end{eqnarray}
	To  compute the inhomogeneous part  $J_e^i$, we use the Green's function  \eqref{eq:G(xi,s)}  and introduce
	\begin{align}
		Q(s):=\Big(\alpha\frac{\partial G(\xi,s)}{\partial \xi}-\beta_{\bm{n},\eta}(\xi)  G(\xi,s)\Big)\Big|_{\xi=0}=\left\{ \begin{aligned}
			&\frac{\int_{s}^{\xi_L} w(t)\,\mathrm{d}t}{\int_{\xi_K}^{\xi_L} w(t)\,\mathrm{d}t},\quad s\geq 0,\\[4pt]
			&-\frac{\int_{\xi_K}^{s} w(t)\,\mathrm{d}t}{\int_{\xi_K}^{\xi_L} w(t)\,\mathrm{d}t},\quad  s\leq 0,
		\end{aligned} \right.
		\label{eq:Q}
	\end{align}
	Note that $Q(s)$ depends on $\eta$ through $\beta_{\bm{n},\eta}(0)$.
	From \eqref{eq:inhomogeneous equation} and the definition \eqref{eq:local} of $S_\eta$, we get
		\begin{align}
		J_e^i 
		 &= \int_{\eta_{e,1}}^{\eta_{e,2}}  \bigg(\alpha\frac{\partial}{\partial \xi}\Big(\int_{\xi_K}^{\xi_L}G(\xi,s)S_{\eta}(s)\;\mathrm{d}s \Big)-\beta_{\bm{n},\eta} \Big(\int_{\xi_K}^{\xi_L}G(\xi,s)S_{\eta}(s)\;\mathrm{d}s \Big)\bigg)\bigg|_{\xi=0} \mathrm{d}\eta \nonumber\\
		 &=\int_{\eta_{e,1}}^{\eta_{e,2}}\int_{\xi_K}^{\xi_L}  \Big(\alpha\frac{\partial G}{\partial \xi}(0,s)-\beta_{\bm{n},\eta}(0)  G(0,s)\Big)S_{\eta}(s) \,\mathrm{d}s \,\mathrm{d}\eta \nonumber\\
		&=\int_{\eta_{e,1}}^{\eta_{e,2}}  \int_{\xi_K}^{\xi_L}Q(s) f \, \mathrm{d}s \mathrm{d}\eta +\int_{\eta_{e,1}}^{\eta_{e,2}} \int_{\xi_K}^{\xi_L} Q(s)\frac{\partial j_{\bm{t}}}{\partial \eta}\; \mathrm{d}s\,\mathrm{d}\eta.  \label{eq:J,i}
	\end{align}

	Integrating \eqref{eq:local eq} over a control volume $V\in\mathcal{T}_h^{*}$ associated with a point $P_V$ and applying the divergence theorem gives
	\begin{eqnarray}\label{eq: integral on control volume}
		\int_{V} -\nabla\cdot \bm{j}\; \mathrm{d}\bm{x} = \sum_{e\subset \partial V}-\int_{e}  \bm{j} \cdot \bm{n}_e \;\mathrm{d}s = 	\int_{V} f  \;\mathrm{d}\bm{x}.
	\end{eqnarray}
	Combining this with the definition of $J_e$, we get
	\begin{eqnarray*}
		\sum_{e\subset \partial V}-J_e =\sum_{e\subset \partial V}-(J_e^h+J_e^i) = \int_{V} f  \;\mathrm{d}\bm{x}.
	\end{eqnarray*}
	Inserting \eqref{eq:J,h} and \eqref{eq:J,i} gives a formulation that incurs no error from the original equation:
	\begin{align}
		K_{V,1}(u) +K_{V,2}(u) =   F_{V,1}(f)  +F_{V,2}(f)  \quad \forall V\in\mathcal{T}_h^{*},\label{eq:exact formulation}
	\end{align}
	with
	\begin{flalign*}
			&K_{V,1}(u) = -\sum_{e\subset \partial V} \int_{\eta_{e,1}}^{\eta_{e,2}} \alpha\, \frac{u_{L_e,\eta}w_e(\xi_{L_e}) - u_{K_{e},\eta}}{\int_{\xi_{K_e}}^{\xi_{L_e}} w_e(t)\,\mathrm{d}t} \;\mathrm{d}\eta, \\
			&K_{V,2}(u) = -\sum_{e\subset \partial V}\int_{\eta_{e,1}}^{\eta_{e,2}}\int_{\xi_{K_e}}^{\xi_{L_e}} Q_e(s)\frac{\partial j_{\bm{t}}}{\partial \eta}\; \mathrm{d}s \, \mathrm{d}\eta ,\\
			&F_{V,1}(f) = \int_{V} f  \;\mathrm{d}\bm{x}, \\
			&F_{V,2}(f) =  \sum_{e\subset \partial V}\int_{\eta_{e,1}}^{\eta_{e,2}}  \int_{\xi_{K_e}}^{\xi_{L_e}}Q_e(s) f(s,\eta)\,\mathrm{d}s\, \mathrm{d}\eta.
	\end{flalign*}
	Here we note that the local coordinate system $(\xi,\eta)$ itself depends on the edge~$e$; in a more precise notation we would have $\xi=\xi(e)$ and $\eta=\eta(e)$; we omit this explicit dependence for brevity. 
	The subscript $e$ on $w$, $Q$, $\xi$, $\eta$ indicates that these quantities are defined in the local coordinate system of edge~$e$. Hence  $w_e$ and  $Q_e$ are the edge-dependent versions of the functions from \eqref{eq,w} and   \eqref{eq:Q},  respectively. 
Relation \eqref{eq:exact formulation}  serves as the starting point for constructing our numerical schemes.
		\begin{remark}\label{remark,1}
			The derivation of the exact relation \eqref{eq:exact formulation} does not require endpoints \(P_{K_e}\) and \(P_{L_e}\) to be
			the centroids of the adjacent control volumes; equivalently, \(\xi_{K_e}\) and \(\xi_{L_e}\) need
			not represent the actual distances from the control volume centers to the edge \(e\). For each
			edge \(e\), the parameters \(\xi_{K_e}\) and \(\xi_{L_e}\) are not fixed by the mesh  while \eqref{eq:exact formulation} remains exact; their difference
			may be a constant multiple of the mesh size, while the constant can be chosen with some freedom. 
		\end{remark}
		\begin{remark}\label{remark,2}
			Another consequence of the freedom emphasized in Remark~\ref{remark,1} concerns edges near
			the domain boundary~\(\partial\Omega\).  The rectangle
			\([\xi_{K_e},\xi_{L_e}]\times[\eta_{e,1},\eta_{e,2}]\) in the local coordinate system
			of an edge~\(e\) may partly extend outside~\(\Omega\) when the edge is adjacent to
			the boundary.  In that case the integrals on 	\([\xi_{K_e},\xi_{L_e}]\times[\eta_{e,1},\eta_{e,2}]\) involving the source term~\(f\) become ill-defined.
			However, by virtue of the flexibility discussed above, the parameters
			\(\xi_{K_e}\) and \(\xi_{L_e}\) can always be chosen so that the rectangle lies
			entirely within~\(\Omega\).  Hence the exact relation \eqref{eq:exact formulation}
			remains meaningful and valid for all edges of the dual mesh.
		\end{remark}
	\section{Construction of complete flux schemes}	\label{sec:4}
	This section develops several complete flux schemes from the exact flux relation \eqref{eq:exact formulation}. 
	Throughout this section, for clarity of exposition, we assume that $\bm{\beta}$ is piecewise constant with respect to each control volume edge $e$. Then, the normal component $\beta_{\bm{n},\eta}(\xi):= \bm{\beta}(\xi,\eta)\cdot \bm{n}_e$ is
	piecewise constant on each edge $e$. We denote this edge-dependent constant by $\beta_{\bm{n}}$.  
	Under this simplification, the edge-dependent weight function $w_e$ and the kernel $Q_{e}$ in the exact  flux relation \eqref{eq:exact formulation} can be evaluated in a closed form:
	\begin{equation}\label{eq: closed form}
		w_e(\xi) = e^{-\frac{\xi-\xi_{K_e}}{\xi_{L_e}-\xi_{K_e}}z_e},\qquad
		Q_{e}(s) =
		\begin{cases}
			\displaystyle\frac{e^{\frac{\xi_{L_e}-s}{\xi_{L_e}-\xi_{K_e}}z_e}-1}{e^{z_e}-1}, & s\ge 0\\[8pt]
			\displaystyle-\frac{1-e^{-\frac{s-\xi_{K_e}}{\xi_{L_e}-\xi_{K_e}}z_e}}{1-e^{-z_e}}, & s\le 0
		\end{cases}.
	\end{equation}
	The quantity $z_e$ appearing above is the local P\'{e}clet number of the edge~$e$, defined by
	\begin{equation}\label{eq:local_Pe}
		z_e := \frac{\beta_{\bm{n}}}{\alpha}(\xi_{L_e}-\xi_{K_e}).
	\end{equation}
		We emphasize that $z_e$ is distinct from the usual mesh P\'{e}clet number: the quantity
		$\xi_{L_e}-\xi_{K_e}$ is a free parameter whose definition is not fixed by the element size~$h$;
		it may be a constant multiple of $h$, although the constant can be freely chosen (cf.\ Remark~\ref{remark,1}).
		Consequently $z_e$ can be adjusted without changing the mesh. 
	
    This section is structured as follows.
	Subsection~\ref{subsec:second_order_2d} presents second-order schemes for general two-dimensional meshes, using a piecewise linear or bi-linear approximation space. 
  Subsection~\ref{subsec:higher_order_rect} 
	first gives the abstract higher-order formulation for an arbitrary polynomial degree~$k$ on general
	meshes, and then specializes it to rectangular meshes where two concrete quadratic ($k=2$)
	approximation spaces: a $C^0$ finite element space and a $C^1$ B-spline space, are described in detail. 
	Finally, Subsection~\ref{subsec:extension_3d} outlines the extension to three dimensions, highlighting
	the treatment of tangential fluxes on polygonal faces.
	\subsection{Second-order complete flux schemes in two dimensions}\label{subsec:second_order_2d}
	The second-order scheme is obtained by inserting a piecewise linear (or bi-linear) numerical
	solution $u_h$ into the  exact flux relation \eqref{eq:exact formulation}:
	\begin{equation}\label{eq:exact formulation,linearuh}
		K_{V,1}(u_h) +K_{V,2}(u_h) =   F_{V,1}(f)  +F_{V,2}(f)  \quad \forall V\in\mathcal{T}_h^{*}.
	\end{equation}
	The control volumes $\mathcal{T}_h^{*}$
	are the standard barycentric dual mesh described in Figure~\ref{fig,0}:
	each control volume is centered at a vertex of the primary mesh, and the control
	volume edges connect the barycenters  of the primary elements to the
	midpoints of their edges.  
	In what follows, we detail the evaluation of each functional in \eqref{eq:exact formulation,linearuh}.

	\paragraph{Normal flux term $K_{V,1}(u_h)$ for linear or bi-linear $u_h$}
	Using the piecewise constant simplification of $\bm{\beta}$, the edge-dependent weight function $w_e$ takes
	the closed form \eqref{eq: closed form}. Substituting it into $K_{V,1}(u_h)$ in
	\eqref{eq:exact formulation} and evaluating the $\eta$-integral gives
	\begin{align}\label{eq,Kv1uh}
		K_{V,1}(u_h) = \sum_{e\subset \partial V} \beta_{\bm{n}}\frac{\overline{u}_{L_e,\eta}   e^{-z_e} - \overline{u}_{K_e,\eta} }{e^{-z_e}-1},
	\end{align}
	where  $z_e$ is the local P\'{e}clet number \eqref{eq:local_Pe}, and 
	\begin{eqnarray*}
		\overline{u}_{K_e,\eta} = \int_{\eta_{e,1}}^{\eta_{e,2}}  u_h(\xi_{K_e},\eta) \;\mathrm{d}\eta,\qquad  \overline{u}_{L_e,\eta} = \int_{\eta_{e,1}}^{\eta_{e,2}}  u_h(\xi_{L_e},\eta) \;\mathrm{d}\eta.
	\end{eqnarray*}
	For a linear or bi-linear $u_h$, the $\eta$-dependence is linear, so the midpoint rule gives
	\begin{equation*}
		\overline{u}_{K_e,\eta} = u_h(\xi_{K_e},0)\,|e|,\qquad
		\overline{u}_{L_e,\eta} = u_h(\xi_{L_e},0)\,|e|,
	\end{equation*}
	where $|e| = \eta_{e,2}-\eta_{e,1}$ is the edge length.  Introducing the Bernoulli function
	\begin{equation}\label{eq:Bernoulli}
		B(z) = \frac{z}{e^{z}-1}\quad (z\neq 0),\qquad B(0) = 1,
	\end{equation}
	we obtain the compact form of \eqref{eq,Kv1uh} as follows:
	\begin{align}
		K_{V,1}(u_h) &= \sum_{e\subset\partial V} -\frac{\alpha|e|}{\xi_{L_e}-\xi_{K_e}}
		\Big(B(z_e)\,u_h(\xi_{L_e},0) - B(-z_e)\,u_h(\xi_{K_e},0)\Big)\nonumber \\
		&=\sum_{e\subset\partial V} -\frac{ \beta_{\bm{n}}|e|}{z_e}
		\Big(B(z_e)\,u_h(\xi_{L_e},0) - B(-z_e)\,u_h(\xi_{K_e},0)\Big)     \label{eq:Kv1_bernoulli}
	\end{align}
	
	\paragraph{Tangential flux $K_{V,2}(u_h)$ for  linear or bi-linear $u_h$}
	Since $\bm{\beta}$ is piecewise constant,  the edge-dependent  kernel $Q_{e}(s)$ defined in \eqref{eq: closed form}
	is independent of $\eta$.  Consequently the double integral in $K_{V,2}(u_h)$
	simplifies by integrating $\partial j_{\bm{t}}/\partial\eta$ in $\eta$ first, yielding
	\begin{align}\label{eq,Kv2 general}
		K_{V,2}(u_h) = -\sum_{e\subset \partial V}\int_{\xi_{K_e}}^{\xi_{L_e}} Q_e(s)\Big( j_{\bm{t}}(s,\eta_{e,2})-j_{\bm{t}}(s,\eta_{e,1})\Big)\; \mathrm{d}s,
	\end{align}
	where  $j_{\bm{t}} = \alpha\partial u_h/\partial\eta - \beta_{\bm{t}} u_h$ with piecewise constant $\beta_{\bm{t}} = \bm{\beta}\cdot\bm{t}_e$. 	For a linear (or bi-linear) $u_h$, the derivative $\partial u_h/\partial\eta$ is constant along
	each edge. Hence the $\alpha\partial u_h/\partial\eta$ part cancels in the difference
	and only $-\beta_{\bm{t}}u_h$ remains:
	\begin{equation}\label{eq,Kv2uh}
		K_{V,2}(u_h) = \sum_{e\subset \partial V}\int_{\xi_{K_e}}^{\xi_{L_e}} Q_e(s)\,\beta_{\bm{t}}\,
		\Big(u_h(s,\eta_{e,2})-u_h(s,\eta_{e,1})\Big)\,\mathrm{d}s.
	\end{equation}
	On each edge, $u_h$ is linear in $s$; then
	\begin{equation}\label{eq: uh linear}
		u_h(s,\eta_{e,i}) = \frac{s-\xi_{K_e}}{\xi_{L_e}-\xi_{K_e}} u_h(\xi_{L_e},\eta_{e,i})
		+ \frac{\xi_{L_e}-s}{\xi_{L_e}-\xi_{K_e}} u_h(\xi_{K_e},\eta_{e,i}),\qquad i=1,2.
	\end{equation}
	Substituting this into \eqref{eq,Kv2uh} shows that $K_{V,2}(u_h)$ involves integrals of
	$Q_e(s)$ multiplied by powers of $s$.  It is therefore natural to introduce the general moments
\begin{equation}\label{eq:Im}
	I_{Q_e,m} := \int_{\xi_{K_e}}^{\xi_{L_e}} Q_e(s)\,s^{m}\,\mathrm{d}s,\qquad m=0,1,2,\dots
\end{equation}
Splitting the integral at $s=0$ and inserting the two branches of $Q_e(s)$ from \eqref{eq: closed form}
leads, after a direct calculation, to the compact representation
	\[
	\begin{aligned}
		I_{Q_e,m} &=
		\frac{(\xi_{L_{e}}-\xi_{K_{e}})^{m+1}}{e^{z_{e}}-1}\left( \frac{e^{m_{L_{e}} z_{e}}}{z_{e}^{m+1}} \int_{0}^{m_{L_{e}} z_{e}} t^{m} e^{-t}\, \mathrm{d}t  -\frac{m_{L_{e}}^{m+1}}{m+1} \right)  \\
		&\quad +\frac{(-1)^{m+1}(\xi_{L_{e}}-\xi_{K_{e}})^{m+1}}{e^{-z_{e}}-1} \left(  \frac{e^{-m_{K_{e}} z_{e}}}{(-z_{e})^{m+1}} \int_{0}^{-m_{K_{e}} z_{e}} t^{m} e^{-t}\, \mathrm{d}t -\frac{m_{K_{e}}^{m+1}}{m+1} \right) ,
	\end{aligned}
	\]
	where  $z_e$ is the local P\'{e}clet number \eqref{eq:local_Pe} and
	\begin{equation*}
		m_{L_e} =  \frac{\xi_{L_e}}{\xi_{L_e}-\xi_{K_e}}, \quad m_{K_e} =  \frac{-\xi_{K_e}}{\xi_{L_e}-\xi_{K_e}}=1-m_{L_e}.
	\end{equation*}
	The remaining integrals are of the elementary type $\int_{0}^{x} t^{m} e^{-t} \mathrm{d}t$; for non‑negative integers $m$, they can be written in a closed form as
	\[
	\int_{0}^{x} t^{m} e^{-t}\, \mathrm{d}t = m! \left( 1 - e^{-x} \sum_{k=0}^{m} \frac{x^{k}}{k!} \right).
	\]
	Substituting this into the expression for $I_{Q_e,m}$ yields explicit formulas in terms of elementary functions.
	For the lowest two moments needed in the second‑order scheme we obtain
	\[
	\begin{aligned}
		I_{Q_e,0}
		&= (\xi_{L_{e}} - \xi_{K_{e}})
		\Big[ B_{1}(z_e, m_{L_{e}}) - B_{1}(-z_e, m_{K_{e}}) \Big], \\[4pt]
		I_{Q_e,1}
		&= (\xi_{L_{e}} - \xi_{K_{e}})^{2}
		\Big[ B_{2}(z_e, m_{L_{e}}) + B_{2}(-z_e, m_{K_{e}}) \Big],
	\end{aligned}
	\]
	with the auxiliary functions
	\begin{equation}\label{eq:Bernoulli,B1}
		B_{1}(z,m) = \frac{e^{mz} - 1 - mz}{z(e^{z} - 1)}\quad (z\neq 0),\qquad B_{1}(0,m) = \frac{m^{2}}{2},
	\end{equation}
	and
	\begin{equation}\label{eq:Bernoulli,B2}
		B_{2}(z,m) = \frac{e^{mz} - 1 - mz - \frac{1}{2}(mz)^{2}}{z^{2}(e^{z} - 1)}\quad (z\neq 0),\qquad B_{2}(0,m) = \frac{m^{3}}{6}.
	\end{equation}
	Now we are ready to obtain $K_{V,2}(u_h)$ in a closed form.
	By   \eqref{eq,Kv2uh} and \eqref{eq: uh linear},  using the
	moments $I_{Q_e,0}$, $I_{Q_e,1}$ gives
	\begin{align}\label{eq, Kv2,linear}
		K_{V,2}(u_h) =&\sum_{e\subset \partial V} \frac{I_{Q_e,1}-\xi_{K_e}I_{Q_e,0}}{\xi_{L_e}-\xi_{K_e}} \beta_{\bm{t}}\bigg( u_h(\xi_{L_e},\eta_{e,2}) - u_h(\xi_{L_e},\eta_{e,1})\bigg) \nonumber\\
		& + \sum_{e\subset \partial V} \frac{\xi_{L_e}I_{Q_e,0}-I_{Q_e,1}}{\xi_{L_e}-\xi_{K_e}} \beta_{\bm{t}}\bigg( u_h(\xi_{K_e},\eta_{e,2}) - u_h(\xi_{K_e},\eta_{e,1})\bigg).
	\end{align}

   \paragraph{Source terms $F_{V,1}$ and $F_{V,2}$}
	The volume integral $F_{V,1}(f)$ is approximated by a quadrature
	rule of sufficient accuracy on the control volume $V$.  The edge-related source term
	$F_{V,2}(f)$
	is evaluated by quadrature on each edge's $(\xi,\eta)$-rectangle. The numerical integration of $f$ against $Q_e(s)$ requires the value of $Q_e(s)$ at the
	quadrature points, and it is important to verify that this evaluation is stable for
	arbitrary choices of the local P\'{e}clet number $z_e $.
 In fact, each branch of $Q_{e}(s)$ in \eqref{eq: closed form} contains a unified
	common part whose stable evaluation is the only ingredient to be checked, namely
	the function
	\begin{equation}\label{eq:Q_exp_form}
		\widetilde Q(a; w) := \frac{e^{a w} - 1}{e^{w} - 1},\qquad a\in[0,1].
	\end{equation}
	For $w=0$, the Taylor expansion $e^{w}=1+w+\frac12 w^2+\cdots$ gives $\widetilde Q(a;0)=a$.
	For $w>0$, the numerator and denominator both grow exponentially; rewriting
	\begin{equation*}
		\widetilde Q(a; w) = \frac{e^{(a-1)w} - e^{-w}}{1 - e^{-w}},
	\end{equation*}
	replaces them by only exponentially decaying quantities, which can be evaluated
	without overflow.  For $w<0$ the original form \eqref{eq:Q_exp_form} is already stable.
	Thus $\widetilde Q(a;w)$ is well-defined and computable for any $w\in\mathbb{R}$, and
	consequently $Q_{e}(s)$ can be evaluated stably at all required points.

At the end of this subsection, we discuss the evaluation of the special functions $B(z)$, $B_1(z,m)$ and
$B_2(z,m)$ that appear in $K_{V,1}(u_h)$ and $K_{V,2}(u_h)$
(see \eqref{eq:Kv1_bernoulli} and \eqref{eq, Kv2,linear}, respectively).  All three can be evaluated stably for arbitrary real arguments,
without suffering from overflow or indeterminate forms when the local
P\'{e}clet number $|z_e|$ is large or near zero.
The following lemma summarizes their basic analytic properties, from which a
simple piecewise evaluation strategy follows.
	\begin{lemma}[Stable evaluation of $B$, $B_1$ and $B_2$]
		\label{lem:Bstable}
		For $m\in[0,1]$ and $z\in\mathbb{R}$, define three functions $B(z)$ as \eqref{eq:Bernoulli}, $B_1(z,m)$ as \eqref{eq:Bernoulli,B1}  and $B_2(z,m)$ as \eqref{eq:Bernoulli,B2}.
		For fixed $m>0$, all three are strictly decreasing in $z$. Their asymptotic behavior is
		\[
		\lim_{z\to+\infty} B(z)=0,\quad \lim_{z\to-\infty}B(z)=-z,
		\]
		\[
		\lim_{z\to+\infty} B_1(z,m)=0,\quad \lim_{z\to-\infty} B_1(z,m)=m,
		\]
		\[
		\lim_{z\to+\infty} B_2(z,m)=0,\quad \lim_{z\to-\infty} B_2(z,m)=\frac{m^2}{2}.
		\]
		Consequently, the functions remain bounded for $z\ge0$ and grow at most linearly for $z<0$. Evaluating them by switching between Taylor series (for $|z|\ll1$), the exact definition (moderate $|z|$), and the asymptotic forms (large $|z|$) guarantees full accuracy and avoids overflow.
	\end{lemma}
	
	\begin{proof}
		Monotonicity follows from $B'(z)=\frac{e^{z}(1-z)-1}{(e^{z}-1)^2}\le0$, with strict inequality for $z\neq0$; the same sign pattern holds for $B_1$ and $B_2$ after similar derivative calculations. For $z\to+\infty$, $e^{z}$ dominates and all three vanish exponentially. For $z\to-\infty$, write $z=-t$ ($t\to+\infty$): $B(-t)\sim t=-z$, $B_1(-t,m)\sim\frac{-1+mt}{t}\to m$, $B_2(-t,m)\sim\frac{-1+mt-\frac12 m^2t^2}{-t^2}\to\frac{m^2}{2}$. The stability claim is immediate from these bounds and the piecewise evaluation strategy.
	\end{proof}
	\subsection{Higher-order complete flux schemes in two dimensions}\label{subsec:higher_order_rect} 
			The higher-order scheme is obtained by inserting a numerical solution $u_h$ of
	polynomial degree~$k$ ($k\ge 2$) into the exact flux relation
	\eqref{eq:exact formulation}.  The scheme takes the same form
	\begin{equation}\label{eq:CF3_scheme}
		K_{V,1}(u_{h}) + K_{V,2}(u_{h}) = F_{V,1}(f) + F_{V,2}(f)\quad \forall V\in\mathcal{T}_{h}^{*},
	\end{equation}
	and is valid on arbitrary two-dimensional meshes.
	We show the evaluation of each functional as follows.
	
	\paragraph{Normal flux term $K_{V,1}(u_h)$ for polynomial solution $u_h$ of degree~$k$}
	The expression \eqref{eq,Kv1uh} of $K_{V,1}(u_h)$
	remains valid.  Using the Bernoulli function $B(z)$ in \eqref{eq:Bernoulli},  we obtain
	\begin{equation*}
		K_{V,1}(u_h) = \sum_{e\subset\partial V} -\frac{ \beta_{\bm{n}}}{z_e}
		\Big(B(z_e)\,\overline{u}_{L_e,\eta} - B(-z_e)\,\overline{u}_{K_e,\eta}\Big).
	\end{equation*}
	The  difference from \eqref{eq:Kv1_bernoulli} is that $u_h(\xi_{K_e},\eta)$ and $u_h(\xi_{L_e},\eta)$ are now polynomials of degree~$k$ in~$\eta$, so the $\eta$-integrals $\overline{u}_{K_e,\eta}$ and $\overline{u}_{L_e,\eta}$ are evaluated by Gaussian
	quadrature.
	
	\paragraph{Tangential flux term $K_{V,2}(u_h)$ for polynomial solution $u_h$ of degree~$k$}
	Formula \eqref{eq,Kv2 general} for $K_{V,2}(u_h)$ remains valid.  For $k\ge 2$, the derivative
	$\partial u_h/\partial\eta$ varies along the edge; consequently the
	$\alpha\partial u_h/\partial\eta$ term in $j_{\bm{t}} = \alpha\partial u_h/\partial\eta
	- \beta_{\bm{t}} u_h$ no longer cancels and both parts contribute. We have
	\begin{align*}
		K_{V,2}(u_h) =& -\sum_{e\subset \partial V}\int_{\xi_{K_e}}^{\xi_{L_e}} Q_e(s)\alpha\Big( \frac{\partial u_h}{\partial \eta}(s,\eta_{e,2})-\frac{\partial u_h}{\partial \eta}(s,\eta_{e,1})\Big) \,\mathrm{d}s\nonumber \\
		&+\sum_{e\subset \partial V}\int_{\xi_{K_e}}^{\xi_{L_e}} Q_e(s)\,\beta_{\bm{t}}\,
		\Big(u_h(s,\eta_{e,2})-u_h(s,\eta_{e,1})\Big)\,\mathrm{d}s.
	\end{align*}
	Substituting the polynomial expansion of $u_h$ in $s$ shows that $K_{V,2}$ involves
	integrals of $Q_e(s)$ multiplied by $s^m$ ($m=0,1,\dots,k$).  These are the moments
	$I_{Q_e,m}$ given in Subsection~\ref{subsec:second_order_2d}, which have closed forms and can be evaluated stably.
	
	\paragraph{Source terms $F_{V,1}$ and $F_{V,2}$}
	The treatment is identical to that of the second-order scheme
	(Subsection~\ref{subsec:second_order_2d}) and is not repeated here.
	
	With all functionals at hand, the scheme \eqref{eq:CF3_scheme} is completely
	determined once a concrete discrete space and a matching control volume partition
	are chosen.  For $k\ge 2$, the standard barycentric dual mesh used in the
	second-order scheme is no longer adequate; the control volumes must be designed
	in accordance with the degrees of freedom of the discrete space, and their
	construction is more varied.  In what follows we present two third-order ($k=2$)
	schemes on rectangular meshes, each with its own dual partition:
	the bi-quadratic $C^{0}$ finite element space
	(Subsubsection~\ref{subsubsec:c0_quad}) and the bi-quadratic $C^{1}$ B-spline space
	(Subsubsection~\ref{subsubsec:c1_bspline}).
	For the $C^{0}$ Lagrange elements, richer control volume constructions of arbitrary
	order on quadrilateral and triangular	meshes  
	are discussed in~\cite{ChenZ2012,LiYangZhangPPR,LinY2015,Wang2016}.
	\subsubsection{Bi‑quadratic \(C^{0}\) finite element space}\label{subsubsec:c0_quad}
	The solution space is taken as the bi‑quadratic \(C^{0}\) finite element space, which corresponds
	to a vertex‑centered finite volume scheme with dual meshes formed piecewise around computing
	nodes.
	\begin{figure}[h!]
		\centering
		\subfigure[$p^*=\frac{1}{4}$] {
			\begin{minipage}[t]{.46\textwidth}
				\centering
				\includegraphics[width=100pt]{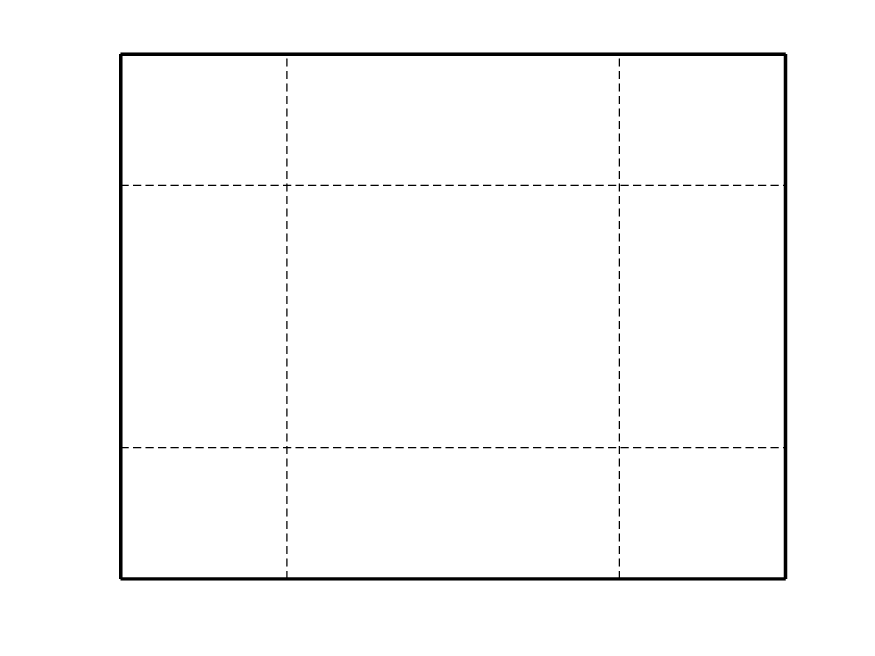}
				\label{fig,2,1}
			\end{minipage}
		}
		\subfigure[$p^*=\frac{1}{3}$]{
			\begin{minipage}[t]{.46\textwidth}
				\centering
				\includegraphics[width=100pt]{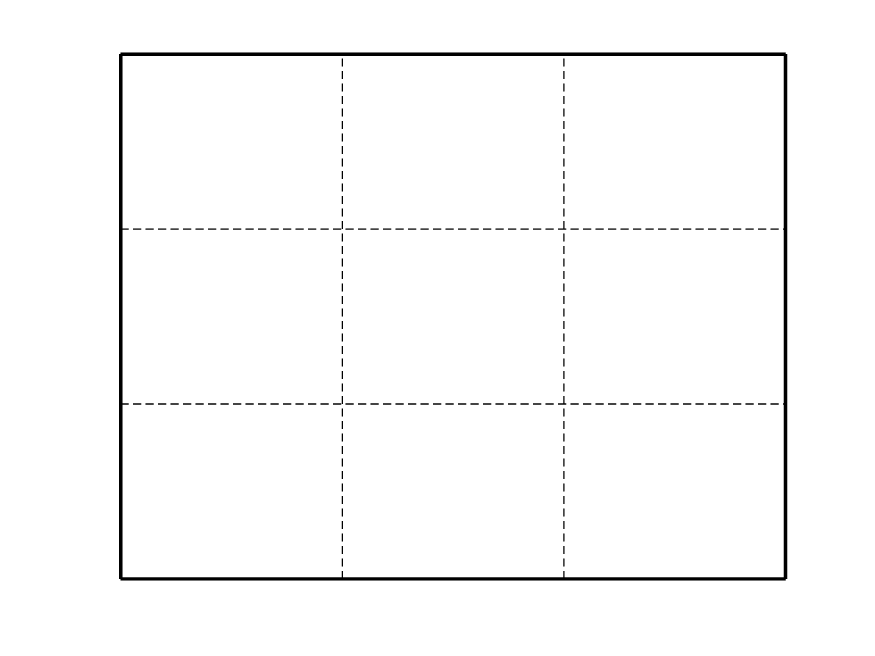}
				\label{fig,2,2}
			\end{minipage}
		}
		\caption{Construction of the element-level dual mesh (control volumes) for a bi-quadratic element, where \(p^*\)  denotes the ratio of the distance from a control volume
			boundary (dashed line) to the nearest element edge  (solid line) to the total span of the
			element in that direction; thus \(0<p^*<1/2\).}
		\label{fig,2}
	\end{figure}
	In this case, each rectangular element possesses nine
	computational nodes: the four vertices, the four edge midpoints, and the element centroid.
	Correspondingly, nine control volumes must be associated with each element. The construction
	of these element-level control volumes is not unique; several variants exist in the literature
	\cite{ZhangZ2015}. A convenient parametrization introduces a scalar \(p^*\) with \(0<p^*<1/2\) that
	specifies the ratio of the distance from a control volume
	boundary  to the nearest element edge to the total span of the
	element in that direction. Two typical choices are \(p^*=1/4\)  and \(p^*=1/3\), which are illustrated in Figure~\ref{fig,2}. Figure~\ref{fig,3,1} displays a rectangular primary mesh on a square domain
	\(\Omega\) together with the corresponding dual mesh (control volume partition) for \(p^*=1/4\).

	\begin{figure}[h!]
		\centering
		\subfigure[Bi-quadratic \(C^0\) finite element space] {
			\begin{minipage}[t]{.46\textwidth}
				\centering
				\includegraphics[width=130pt]{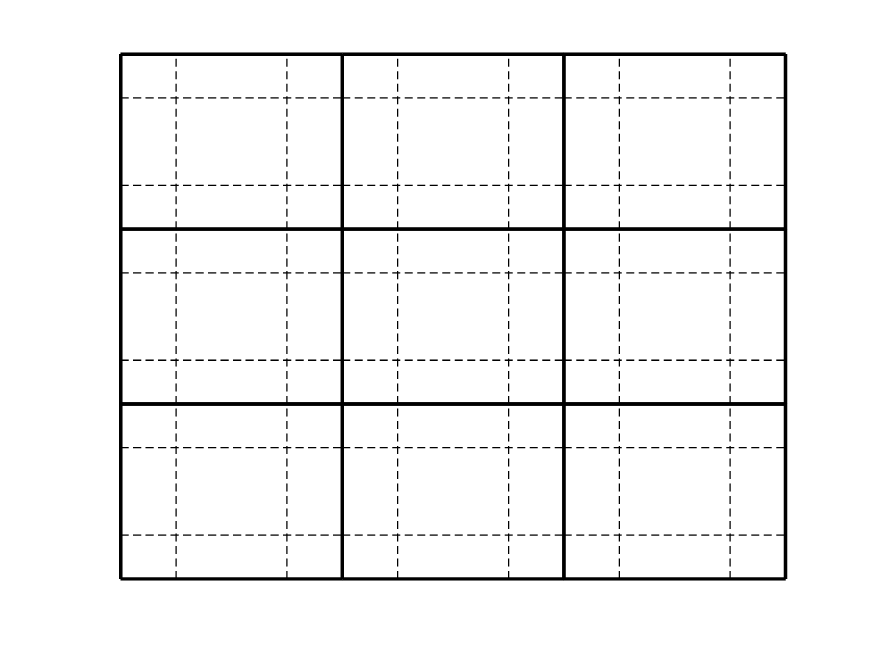}
				\label{fig,3,1}
			\end{minipage}
		}
		\subfigure[Bi-quadratic \(C^1\) B-spline space]{
			\begin{minipage}[t]{.46\textwidth}
				\centering
				\includegraphics[width=130pt]{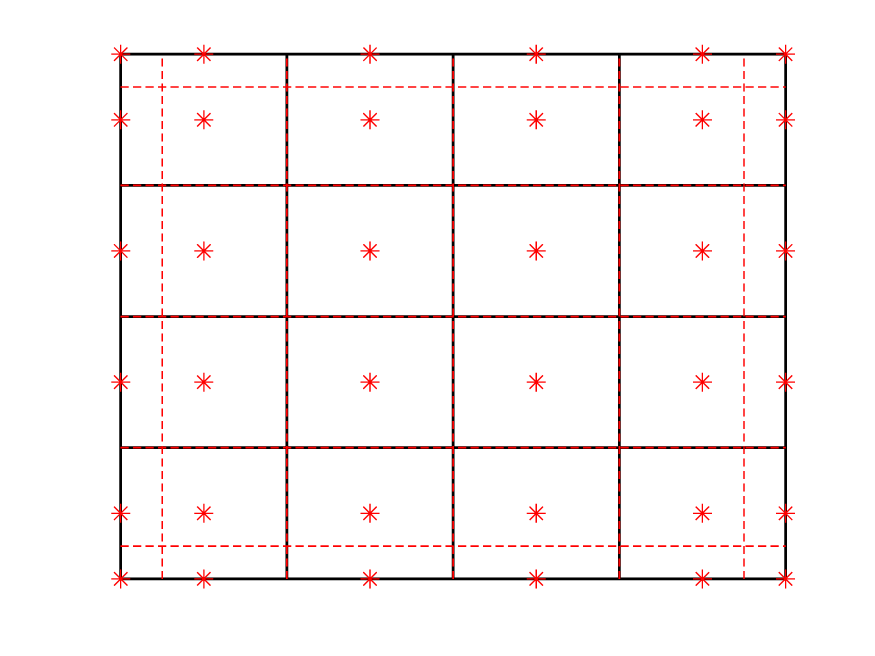}
				\label{fig,3,2}
			\end{minipage}
		}
		\caption{Primary rectangular meshes and corresponding dual meshes for different function spaces. The red $``\ast$" denotes the Greville abscissae for the bi-quadratic $C^1$ B-spline space.}
		\label{fig,3}
	\end{figure}
	\subsubsection{Bi‑quadratic \(C^{1}\) B‑spline space}\label{subsubsec:c1_bspline}
We consider a rectangular domain \(\Omega=[a,b]\times[c,d]\). 
	For the bi-quadratic \(C^1\) B-spline space, the primary mesh originates from a parametric
	construction. Let a strictly increasing sequence \(a = x_1 < x_2 < \dots < x_m = b\) be given.
	By repeating the first and last knots to multiplicity three, we obtain the augmented knot
	vector
	\[
	\Xi = \{x_1,x_1,x_1, x_2, x_3, \dots, x_{m-1}, x_m,x_m,x_m\}:= \{t_1, t_2, \dots, t_{m+4}\},
	\]
	which generates the one-dimensional quadratic \(C^1\) B-spline basis
	\(\{B_{i}(x)\}_{i=1}^{m+1}\) via the standard Cox--de Boor recurrence \cite{HughesCottrellBazilevs2005}.
	The support of each \(B_i\) is the interval \([t_i, t_{i+3}]\). An analogous construction
	in the second coordinate direction with knots \( c=y_1=y_1=y_1<y_2< \dots <y_{n-1}< y_{n}=y_{n}=y_{n} = d \) or  \(c = s_1=s_2=s_3<s_4< \dots <s_{n+1}< s_{n+2}=s_{n+3}=s_{n+4} = d\)  yields the basis
	\(\{B_{j}(y)\}_{j=1}^{n+1}\). The support of each \(B_j\) is the interval \([s_j, s_{j+3}]\). The tensor-product space on \(\Omega\) is spanned by
	\(\{B_i(x)B_j(y)\}\). The primary mesh (solid lines in Figure~\ref{fig,3,2}) consists of the knot lines
	\(x= x_k\) and \(y = y_l\); the solution is a polynomial within each rectangle and is
	\(C^1\) across the knot lines.
	To construct the dual mesh (control volumes), we associate with each tensor‑product
	basis function \(B_i(x)B_j(y)\) a rectangular control volume whose center is
	located at the \emph{Greville abscissae} \((t_i^*,s_j^*)\) ~\cite{KamberG2022,KamberG2020}, defined by
	\[
	t_i^* = \frac{t_{i+1} + t_{i+2}}{2},
	\qquad
	s_j^* = \frac{s_{j+1} + s_{j+2}}{2}.
	\]
	The Greville abscissa is the average of the two interior knots, and these points serve as the natural collocation nodes for
	the control volume scheme.  The boundaries of the control volumes (dashed
	lines in Figure~\ref{fig,3,2}) are placed midway between adjacent Greville abscissae in each
	direction. Because some of these boundaries coincide with the edges of the primary mesh, the dashed lines are colored red in the figure to distinguish them clearly from the primary mesh lines.  This dual partition guarantees that the number of control volumes equals the
	dimension of the approximation space, and is particularly well suited for higher-order complete flux discretizations.

	\subsection{Extension to three dimensions}\label{subsec:extension_3d}
	The extension of the  two-dimensional complete flux scheme to three-dimensional problems is conceptually straightforward yet notationally more involved. In the following we present the key steps, focusing on the modifications required when the control volume boundary $e$ (now a polygonal face) is treated.
	
	First, we introduce the local coordinate system in three dimensions. Let $e$ be a polygonal face of a control volume $V \in \mathcal{T}_h^*$ (the dual mesh). We introduce a local orthonormal coordinate system $(\xi,\eta_1,\eta_2)$ with origin at the centroid of $e$:
	\begin{itemize}
		\item[-] the $\xi$ axis is aligned with the unit normal $\bm{n}_e$ pointing outward from $V$;
		\item[-] $\bm{\eta}:=(\eta_1,\eta_2)$ denotes the orthogonal tangential coordinates on $e$.
	\end{itemize}
	Thus any point on $e$ has $\xi=0$, while the two neighboring control volume centers $P_K$ and $P_L$ lie at $\xi = \xi_K \leq 0$ and $\xi = \xi_L \geq 0$, respectively. The face $e$ is parameterized by $(\eta_1,\eta_2)$ over a two-dimensional domain.
	
	The total flux vector is $\bm{j} = \alpha\nabla u - \bm{\beta}u$. Its normal and tangential components with respect to $e$ are
	\begin{equation*}
		j_{\bm{n}} = \bm{j}\cdot\bm{n}_e = \alpha\frac{\partial u}{\partial\xi} - \beta_{\bm{n}} u,\qquad
		\bm{j}_{\bm{t}} = \alpha\nabla_{\!\bm{t}}u - \bm{\beta}_{\bm{t}} u,
	\end{equation*}
	where $\beta_{\bm{n}} = \bm{\beta}\cdot\bm{n}_e$, $\bm{\beta}_{\bm{t}}= \bm{\beta} - \beta_{\bm{n}}\bm{n}_e$, and $\nabla_{\!\bm{t}} = (\partial/\partial\eta_1,\partial/\partial\eta_2)^{T}$ is the tangential gradient. Note that $\bm{j}_{\bm{t}} $ is a two-dimensional vector field defined on the face $e$.
	Following a derivation completely analogous to the two-dimensional case \eqref{eq:exact formulation}, we obtain a similar formulation  that incurs no error from the original equation: 
	\begin{align}\label{eq:exact formulation,3d}
		K_{V,1}(u) +K_{V,2}(u) =   F_{V,1}(f)  +F_{V,2}(f)  \quad \forall V\in\mathcal{T}_h^{*},
	\end{align}
	with
	\begin{align*}
			K_{V,1}(u) &=-\sum_{e\subset \partial V} \int_{e}  \alpha\, \frac{u_{L_e,\bm{\eta}}w_e(\xi_{L_e}) - u_{K_{e},\bm{\eta}}}{\int_{\xi_{K_e}}^{\xi_{L_e}} w_e(t)\,\mathrm{d}t} \,\mathrm{d}\bm{\eta}  \\
			K_{V,2}(u) &=-\sum_{e\subset \partial V} \int_{e}\int_{\xi_{K_e}}^{\xi_{L_e}}Q_e(s) \nabla_{\bm{t}}\cdot \bm{j}_{\bm{t}}(s,\bm{\eta})\,\mathrm{d}s\,\mathrm{d}\bm{\eta} \\
			F_{V,1}(f) &= \int_{V} f \, \mathrm{d}\bm{x} \\
			F_{V,2}(f) &= \sum_{e\subset \partial V}  \int_{e}\int_{\xi_{K_e}}^{\xi_{L_e}}Q_e(s) f(s,\bm{\eta})\, \mathrm{d}s\,\mathrm{d}\bm{\eta}
	\end{align*}
	where $u_{K_e,\bm{\eta}}$ and $u_{L_e,\bm{\eta}}$ are  the traces of $u$ along the planes $\xi=\xi_K$ and $\xi=\xi_L$, respectively, and $\nabla_{\bm{t}} \cdot$ denotes the tangential divergence on $e$.


	We present here the second-order discretization; higher-order variants can be derived analogously, though they are notationally more involved, and are omitted for brevity. The control
	volume partition for the second-order scheme follows the standard construction: on
	hexahedral meshes, each hexahedron is divided into eight sub-volumes by connecting its
	centroid to the face centroids and edge midpoints; on tetrahedral meshes, each tetrahedron
	is divided into four sub-volumes by connecting its barycenter to the face barycenters and
	edge midpoints, which is a direct three-dimensional analogue of the barycentric construction. For hexahedral meshes, the construction of control volumes for  higher-order elements follows as a direct generalization of the two-dimensional construction. For tetrahedral meshes, the construction of control volumes for arbitrary higher-order finite volume schemes can be found in~\cite{LiYangZhangPPR}. 
	
	Let $u_h$ be a piecewise tri-linear approximation on a hexahedral mesh, or a piecewise linear approximation on a tetrahedral mesh. The second-order scheme is obtained by inserting $u_h$ into the  exact flux relation \eqref{eq:exact formulation,3d}. In what follows we detail the evaluation of each functional.
	
	\paragraph{Normal flux term $K_{V,1}(u_h)$ for linear or tri-linear $u_h$}
	The expression $K_{V,1}(u_h)$
	from \eqref{eq:exact formulation,3d} reduces exactly to the same form as in two
	dimensions, with the edge integral replaced by a face integral.  
	
		\paragraph{Tangential flux term $K_{V,2}(u_h)$ for linear or tri-linear $u_h$}

The treatment of the tangential contribution from \eqref{eq:exact formulation,3d} involves the tangential divergence of $\bm{j}_{\bm{t}} = \alpha\nabla_{\!\bm{t}}u_h - \bm{\beta}_{\bm{t}} u_h$. Applying the tangential divergence theorem on the face $e$ and noting that $Q_e(s)$ is independent of $\bm{\eta}$ gives
\begin{equation*}
	K_{V,2}(u_h) = -\sum_{e\subset\partial V} \int_{\xi_{K_e}}^{\xi_{L_e}} Q_e(s) \Big(\int_{\partial e} \bigl( \alpha\nabla_{\!\bm{t}}u_h - \bm{\beta}_{\bm{t}} u_h \bigr)\cdot\bm{n}_{\partial e}\,\mathrm{d}l\Big) \,\mathrm{d}s.
\end{equation*}
For a linear  (or tri-linear) $u_h$, the tangential gradient term integrates to zero on each face, so only the convection part survives.  Let $\partial e$ consist of $m$ edges $\gamma_1,\dots,\gamma_m$ with vertices $\bm{v}_1,\dots,\bm{v}_m$ in counterclockwise order.  On each edge $\gamma_j = \overline{\bm{v}_j\bm{v}_{j+1}}$,
\[
\int_{\gamma_j} (\bm{\beta}_{\bm{t}}\cdot\bm{n}_{\partial e})\,u_h\,\mathrm{d}l
= (\bm{\beta}_{\bm{t}}\cdot\bm{n}_{\partial e})\,|\gamma_j|\,
\frac{u_h(s,\bm{v}_j) + u_h(s,\bm{v}_{j+1})}{2},
\]
where $|\gamma_j|$ is the length of $\gamma_j$.  Assembling the edge contributions and inserting the linear interpolation of $u_h$ in $s$, the integration against $Q_e(s)$ reduces to the moments $I_{Q_e,0}, I_{Q_e,1}$ from \eqref{eq:Im}.  The final expression reads
\begin{align*}
	K_{V,2}(u_h)
	=&\sum_{e\subset \partial V} \sum_{j=1}^{m} \frac{I_{Q_e,1}-\xi_{K_e}I_{Q_e,0}}{2(\xi_{L_e}-\xi_{K_e})} |\gamma_j|(\bm{\beta}_{\bm{t}}\cdot\bm{n}_{\partial e,\gamma_j})
	\big(u_h(\xi_{L_e},\bm{v}_j)+u_h(\xi_{L_e},\bm{v}_{j+1})\big)\\
	+& \sum_{e\subset \partial V} \sum_{j=1}^{m} \frac{\xi_{L_e}I_{Q_e,0}-I_{Q_e,1}}{2(\xi_{L_e}-\xi_{K_e})} |\gamma_j|(\bm{\beta}_{\bm{t}}\cdot\bm{n}_{\partial e,\gamma_j})
	\big(u_h(\xi_{K_e},\bm{v}_j)+u_h(\xi_{K_e},\bm{v}_{j+1})\big).
\end{align*}
		\paragraph{Source terms $F_{V,1}$ and $F_{V,2}$}
	The volume integral $F_{V,1}(f)$ is approximated by a standard
	quadrature rule on the control volume $V$.  The face-related source term
	$F_{V,2}(f)$
	is evaluated by quadrature on each face's $(\xi,\bm{\eta})$-prism.  The weight function
	$Q_e(s)$ is one-dimensional (it depends only on $\xi$) and its stable evaluation for
	arbitrary P\'{e}clet number $z_e$ follows the discussion in
	Subsection~\ref{subsec:second_order_2d}.

	\section{Numerical experiments}\label{sec:numerical}
	The accuracy and robustness of the complete flux schemes are examined through two numerical tests.
	The first test (Section~\ref{subsec:2D_convergence}) verifies the convergence rates in two and three dimensions with manufactured solutions and constant convection, covering the second-order schemes on triangles, rectangles, hexahedra, and tetrahedra as well as the third-order schemes on rectangles.
	The second test (Section~\ref{subsec:pn_junction}) treats a decoupled drift-diffusion model for a GaAs p-n~junction in two dimensions, demonstrating the applicability of the complete-flux scheme to semiconductor device simulation.
	
	\subsection{Convergence order verification}
	\label{subsec:2D_convergence}
	We numerically verify the convergence rates of the six considered schemes using manufactured solutions on the unit square $\Omega=(0,1)^2$ (two dimensions) and on the unit cube $\Omega=(0,1)^3$ (three dimensions).
\begin{example}\label{example,1}
	Let $d\in\{2,3\}$ denote the spatial dimension. Consider the steady-state convection-diffusion equation \eqref{eq:model} with constant velocity field $\bm{\beta}\in\mathbb{R}^d$,
	namely $\bm{\beta}=(50,60)^{T}$ in two dimensions and $\bm{\beta}=(10,15,20)^{T}$ in three dimensions,
	and two representative diffusion coefficients: $\alpha=10^{-1}$ (mildly convection-dominated)
	and $\alpha=10^{-8}$ (strongly convection-dominated).
	The source term $f$ and the homogeneous Dirichlet boundary conditions are chosen so that the exact solution is
	\begin{equation}
		u(\bm{x}) = \Big(\prod_{i=1}^{d} x_i\Big)
		\prod_{i=1}^{d} \frac{1-e^{\beta_i(x_i-1)}}{1-e^{-\beta_i}},
	\end{equation}
	where $\bm{x}=(x_1,\ldots,x_d)\in(0,1)^d$ and $\beta_i$ denotes the $i$-th component of $\bm{\beta}$.
	This solution vanishes on the whole boundary and develops sharp exponential layers near $x_i=1$
	when the components of $\bm{\beta}$ are large enough, see Figure~\ref{fig:exact} and Figure~\ref{fig:exact,2}.
\end{example}
	The performance of six distinct complete flux (CF) schemes is evaluated:
	\begin{itemize}
		\item \textbf{CF2-Tri}:  The second-order CF scheme on triangular meshes (piecewise linear approximation), see Subsection~\ref{subsec:second_order_2d};
		\item \textbf{CF2-Rect}:  The second-order CF scheme on rectangular meshes (piecewise bi-linear approximation), see Subsection~\ref{subsec:second_order_2d};
		\item \textbf{CF2-Hex}: The second-order CF scheme on hexahedral meshes in three dimensions (piecewise tri-linear approximation), see Subsection~\ref{subsec:extension_3d};
		\item \textbf{CF2-Tet}: The second-order CF scheme on tetrahedral meshes in three dimensions (piecewise linear approximation), see Subsection~\ref{subsec:extension_3d};
		\item \textbf{CF3-RectC0}:  The third-order CF scheme on rectangular meshes using the bi-quadratic $C^{0}$ finite element space, with dual control volumes defined by $p^{*}=1/4$ (cf.\ Figure~\ref{fig,3,1}), see Subsection~\ref{subsec:higher_order_rect};
		\item \textbf{CF3-RectC1}: The third-order CF scheme on rectangular meshes using the bi-quadratic $C^{1}$ B-spline space, with dual partitions constructed from the Greville abscissae (cf.\ Figure~\ref{fig,3,2}), see Subsection~\ref{subsec:higher_order_rect}.
	\end{itemize}
	As pointed out in Remark~\ref{remark,1}, the parameters $\xi_{K_e}$ and $\xi_{L_e}$ associated 
	with each control volume edge can be chosen independently while keeping the flux 
	relation exact. In the present subsection, we consistently adopt the standard choice: 
	$\xi_{K_e}$ and $\xi_{L_e}$ are taken as the actual designated‌ distances from the edge $e$ 
	to the adjacent control volume centers. For a simple verification of the flexibility of $\xi_{K_e}$ and $\xi_{L_e}$, we also test an 
	alternative variant of the CF2-Rect scheme. In this variant, for each edge $e$, we replace 
	the geometric parameters by random perturbations 
	\[
	\tilde{\xi}_{K_e} = \xi_{K_e}\bigl(0.1 + 0.8 r_e\bigr),\qquad
	\tilde{\xi}_{L_e} = \xi_{L_e}\bigl(0.1 + 0.8 r_e\bigr),
	\]
	where $r_e$ is a random number drawn uniformly from $(0,1)$, independently for each edge $e$.
	\begin{figure}[htbp]
		\centering
		\subfigure[Exact solution for $d=2$] {
			\begin{minipage}[b]{0.45\textwidth}
				\centering
				\includegraphics[width=\linewidth]{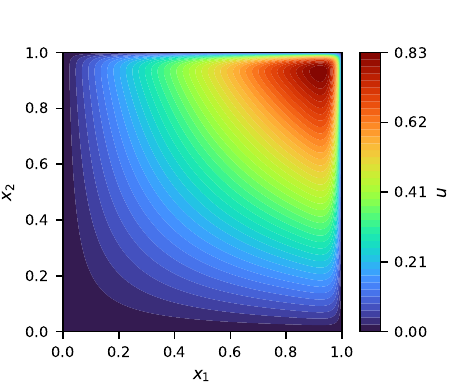}
				\label{fig:exact}
			\end{minipage}
		}
		\subfigure[Exact solution for $d=3$] {
			\begin{minipage}[b]{0.45\textwidth}
				\centering
				\includegraphics[width=\linewidth]{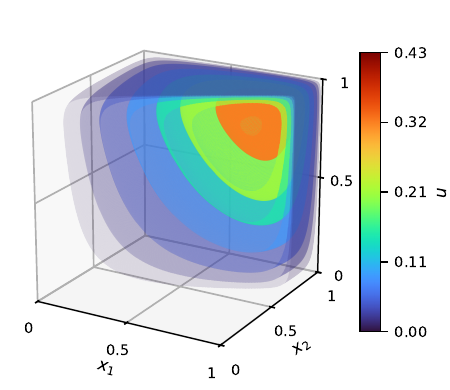}
				\label{fig:exact,2}
			\end{minipage}
		}\\
		\subfigure[Initial triangular mesh] {
			\begin{minipage}[b]{0.45\textwidth}
				\centering
				\includegraphics[width=130pt]{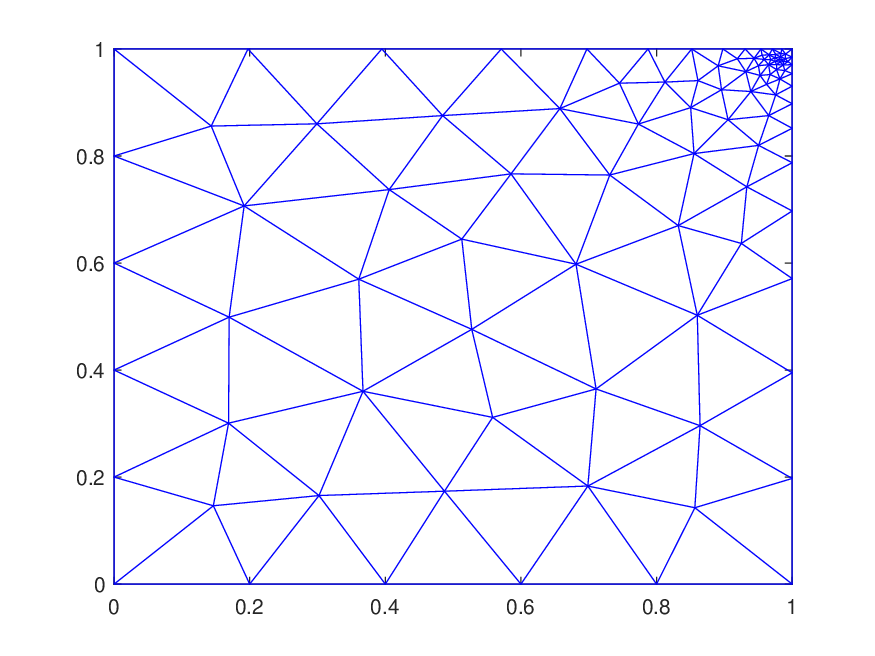}
				\label{fig:mesh}
			\end{minipage}
		}
		\caption{Exact solutions and a locally refined triangular mesh for Example~\ref{example,1}.}
		\label{fig:solution_and_mesh}
	\end{figure}
	All rectangular and hexahedral meshes are uniform Cartesian grids: in two dimensions the unit square $\Omega=(0,1)^2$ is subdivided into $N^2$ congruent squares, and in three dimensions the unit cube $\Omega=(0,1)^3$ is subdivided into $N^3$ congruent cubes, for the same number of intervals $N$ in each direction.  For the triangular meshes, we start from an initial grid with 170 triangles that is locally refined near the point
	$(0.98,0.98)$, i.e., in the vicinity of the corner  $(1,1)$ where the sharp boundary layers are most pronounced (see Figure~\ref{fig:mesh}); the subsequent triangular grids are obtained by uniform refinement of this initial mesh.  Unlike the locally refined triangular meshes, the tetrahedral meshes are uniform: each cube of the same Cartesian grid is subdivided into six tetrahedra, and the meshes are refined by uniform bisection of the one-dimensional node sequences; this is sufficient since the boundary layers of the three-dimensional example are milder than in two dimensions.
	
		\begin{table}[h!]
		\centering
		\caption{Discrete maximum error and convergence order for Example~\ref{example,1} in two dimensions with $\bm{\beta}=(50,60)^{T}$ and $\alpha=10^{-1},10^{-8}$.}
		\label{tab:conv_linfty_combined}
		\begin{tabular}{c@{\quad}c@{\quad}*{2}{c}@{\quad}*{2}{c}}
			\hline
			\multirow{2}{*}{Scheme}  & \multirow{2}{*}{Dof}& \multicolumn{2}{c}{$\alpha=10^{-1}$} & \multicolumn{2}{c}{$\alpha=10^{-8}$} \\
			\cline{3-4} \cline{5-6}
			&      & $\|u-u_h\|_{\ell^{\infty}}$ & Rate& $\|u-u_h\|_{\ell^{\infty}}$ & Rate \\
			\hline
			\multirow{4}{*}{CF2-Rect}  
			& 9409  &4.4780e-03  & -- & 8.3455e-04 & -- \\
			& 37249  &5.7713e-04  & 2.9781& 2.5104e-03 &-1.6007 \\
			& 148225 & 7.0555e-05 & 3.0435 & 1.1186e-03 &1.1706 \\
			& 591361 &1.3737e-05& 2.3651& 3.5915e-04&1.6421 \\
			& 2362369  & 3.1825e-06& 2.1118 & 1.0108e-04& 1.8308 \\
			\hline
			\multirow{4}{*}{\makecell[c]{CF2-Rect\\ (random $(\xi_{K_{e}},\xi_{L_{e}})$)}}  
			& 9409  &4.9993e-02  & -- & 7.2598e-02 & -- \\
			& 37249  &1.1260e-02  & 2.1667& 1.7711e-02 &2.0506 \\
			& 148225 &2.3916e-03 & 2.2436 & 5.7124e-03 &1.6386 \\
			& 591361 &4.9624e-04& 2.2731& 1.7253e-03&1.7305 \\
			& 2362369  & 1.0605e-04& 2.2223& 4.9098e-04& 1.8148\\
			\hline
			\multirow{4}{*}{CF2-Tri}    
			& 5578  & 1.3964e-02 & -- & 2.2567e-02 & -- \\
			& 22034 & 4.9049e-03 & 1.5232 & 8.1318e-03& 1.4860 \\
			& 87586 & 1.4053e-03 & 1.8116 & 2.7328e-03 &1.5803 \\
			& 349250&3.6543e-04  & 1.9476&8.6833e-04 & 1.6578 \\
			& 1394818 &9.2750e-05 & 1.9804 & 2.5712e-04  & 1.7578 \\
			\hline
		\end{tabular}
	\end{table}
	\begin{table}[h!]
		\centering
		\caption{Discrete maximum error and convergence order for Example~\ref{example,1} in three dimensions with $\bm{\beta}=(10,15,20)^{T}$ and $\alpha=10^{-1},10^{-8}$.}
		\label{tab:conv_3d}
		\begin{tabular}{c@{\quad}c@{\quad}*{2}{c}@{\quad}*{2}{c}}
			\hline
			\multirow{2}{*}{Scheme}  & \multirow{2}{*}{Dof}& \multicolumn{2}{c}{$\alpha=10^{-1}$} & \multicolumn{2}{c}{$\alpha=10^{-8}$} \\
			\cline{3-4} \cline{5-6}
			&      & $\|u-u_h\|_{\ell^{\infty}}$ & Rate& $\|u-u_h\|_{\ell^{\infty}}$ & Rate \\
			\hline
			\multirow{5}{*}{CF2-Hex}  
			& 35937 & 2.7938e-03 & -- & 6.5910e-03 & -- \\
			& 274625 & 3.7102e-04 & 2.9127 & 2.9990e-03 & 1.1360 \\
			& 912673 & 1.0509e-04 & 3.1111 & 1.4757e-03 & 1.7490 \\
			& 2146689 & 4.9837e-05 & 2.5933 & 8.5483e-04 & 1.8979 \\
			& 4173281 & 2.9400e-05 & 2.3651 & 5.5231e-04 & 1.9575 \\
			\hline
			\multirow{5}{*}{CF2-Tet}
			& 35937 & 7.4755e-03 & -- & 8.2370e-03 & -- \\
			& 274625 & 3.0286e-03 & 1.3035 & 2.1796e-03 & 1.9181 \\
			& 912673 & 1.6000e-03 & 1.5737 & 8.1476e-04 & 2.4268 \\
			& 2146689 & 9.7290e-04 & 1.7293 & 3.9559e-04 & 2.5115 \\
			& 4173281 & 6.3315e-04 & 1.9251 & 2.6063e-04 & 1.8700 \\
			\hline
		\end{tabular}
	\end{table}
	For each scheme, the error is measured in the discrete maximum norm taken over all control volume centers (computational nodes):
	\[
	\|u-u_h\|_{\ell^{\infty}} \;=\; \max_{P\in\mathcal{N}_h} \bigl|u(P)-u_h(P)\bigr|,
	\]
	where $\mathcal{N}_h$ denotes the set of nodes associated with the respective discrete space:
	the vertices for the piecewise linear, bi-linear, or tri-linear spaces, the nine computing nodes of
	each bi-quadratic $C^0$ element, and the Greville abscissae for the bi-quadratic $C^1$ B-spline space. 
	The observed convergence orders, as confirmed in Tables~\ref{tab:conv_linfty_combined} and~\ref{tab:conv_3d} and visualized in Figure~\ref{fig:convergence,alpha}, demonstrate the expected second-order accuracy for the CF2-Tri, CF2-Rect, CF2-Hex, and CF2-Tet schemes and the fourth-order accuracy for the CF3-RectC0 and CF3-RectC1 schemes, both of which exhibit superconvergence, at the computing nodes (CF3-RectC0) and at the Greville abscissae (CF3-RectC1).  Moreover,  Figure~\ref{fig:convergence,alpha} shows that the $C^1$ B-spline solution is more accurate per degree of freedom than its $C^0$ Lagrange counterpart, a well-known advantage often emphasized in the spline community~\cite{HughesCottrellBazilevs2005}.  For $\alpha=10^{-1}$, however, this advantage fades under mesh refinement, as the $C^0$ Lagrange solution exhibits an even more remarkable ``ultra-convergence'' behavior, attaining a convergence order close to five.
		\begin{figure}[h!]
		\centering
		\subfigure[$\alpha =10^{-1}$] {
			\begin{minipage}[b]{0.45\textwidth}
				\centering
				\includegraphics[width=160pt]{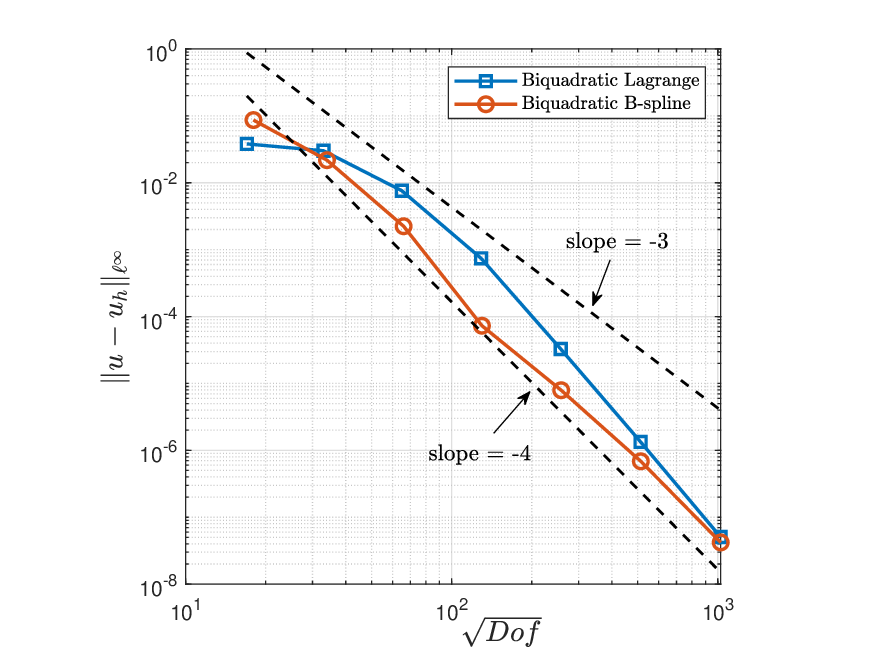}
				\label{fig:converg1}
			\end{minipage}
		}
		\hfill
		\subfigure[$\alpha =10^{-8}$] {
			\begin{minipage}[b]{0.45\textwidth}
				\centering
				\includegraphics[width=160pt]{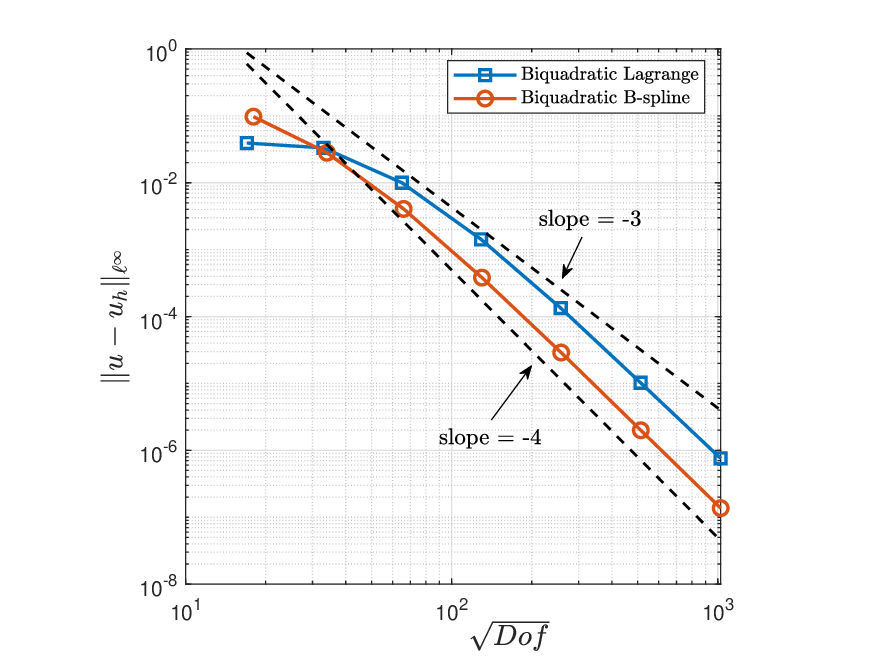}
				\label{fig:converg2}
			\end{minipage}
		}
		\caption{Discrete maximum error $\|u-u_h\|_{\ell^{\infty}}$ versus degrees of freedom (Dof) for
			Example~\ref{example,1} with $\bm{\beta}=(50,60)^{T}$ and $\alpha=10^{-1},10^{-8}$. The discrete solution spaces are
			taken as bi-quadratic Lagrange and B-spline basis spaces.}
		\label{fig:convergence,alpha}
	\end{figure}
	
	At the end of this subsection, we highlight a crucial observation concerning higher-order finite volume methods on rectangular meshes.
	Classical finite volume schemes that are obtained by inserting the primitive flux definition \eqref{flux,j} directly into the balance relation \eqref{eq: integral on control volume} are known to impose stringent constraints on the dual grid for quadratic trial spaces.
	For the bi-quadratic $C^0$ finite volume method, the optimal $L^2$
	convergence order can be achieved only when the control volumes are constructed
	using the Gauss points of the elements~\cite{LinY2015,ZhangZ2015}; in the parametrization of Subsection~\ref{subsec:higher_order_rect} this
	corresponds to a very specific choice of $p^* = 1/2 - \sqrt{3}/6$, and any other value degrades the $L^2$ accuracy. Similarly, for the bi-quadratic $C^1$ B-spline space, the standard
	finite volume scheme that places the control volume centers at the Greville abscissae
	is observed to deliver suboptimal second-order $L^2$ convergence~\cite{KamberG2022,KamberG2020}, and the underlying
	mechanism is still not fully understood.
	
	In contrast, the complete flux schemes developed in the present work rely on a
	fundamentally different flux reconstruction (Sections~\ref{sec,3} and~\ref{sec:4}). This reconstruction
	renders the resulting discrete scheme essentially insensitive to the geometric
	parameters that define the dual mesh. Concretely, the CF3-RectC0 scheme maintains its
	optimal convergence order for any admissible value of the parameter $p^*$; the choice
	$p^* = 1/4$ reported in Table~\ref{tab:conv_linfty_combined} is just one convenient instance, and other values also
	yield the optimal convergence order provided that $\xi_{K_e} + \xi_{L_e} = 0$ for interior edge, i.e., the local region around the edge~$e$, whose extent in the normal direction is determined by the parameters $\xi_{K_e}$ and $\xi_{L_e}$, is chosen symmetrically with respect to the edge.   Likewise, the CF3-RectC1 scheme retains optimal  accuracy when the control volumes are
	centered at the Greville points. Hence, the complete flux approach eliminates the traditional dependence on
	the precise positioning of control volume boundaries; a rigorous theoretical explanation of this robustness is left for future work.	
		
	\subsection{Decoupled drift-diffusion model for a p-n junction}
		\label{subsec:pn_junction}
		
		The complete flux framework of Section~\ref{sec:4} can be applied to the van~Roosbroeck drift-diffusion system that governs semiconductor devices~\cite{Selberherr1984}. 
		In this test, we consider a GaAs p-n junction at thermal equilibrium ($T=300\,\mathrm{K}$)
		on the two-dimensional rectangular domain $\Omega=[0,R]\times[0,L]$ with radial extent
		$R=0.5\,\mu\mathrm{m}$ and axial length $L=3\,\mu\mathrm{m}$, where $r\in[0,R]$ and $z\in[0,L]$
		are the radial and axial coordinates; the junction is at $z=L/2$.
		Doping is abrupt and depends only on the axial coordinate:
		\begin{equation}\label{eq:doping_pn}
			C(z) = \begin{cases}
				-N_A, & z < L/2 \quad\text{(p-region)},\\[2pt]
				+N_D, & z \ge L/2 \quad\text{(n-region)},
			\end{cases}
		\end{equation}
		together with acceptor density $N_A$ and donor density $N_D$.
		Following the Debye scaling with reference concentration $C_0=N_A$ and length scale $\ell=1\,\mu\mathrm{m}$, the dimensionless Poisson--Boltzmann equation reads
		\begin{equation}\label{eq:poisson_pb}
			-\lambda^2\,\Delta\bar\psi = \bar p - \bar n + \bar C,\qquad
			\lambda^2 = \frac{\varepsilon_s U_T}{q\,C_0\,\ell^2},
		\end{equation}
		where $\bar\psi=\psi/U_T$ is the scaled electrostatic potential, $\bar C = C/C_0$,
		$\bar p = \bar N_v\,e^{-\bar\psi}$, $\bar n = \bar N_c\,e^{\bar\psi-\bar E_c}$,
		$\bar N_v=N_v/C_0$, $\bar N_c=N_c/C_0$, $\bar E_c=E_c/U_T$, and
		$U_T=k_BT/q$ is the thermal voltage.
		Dirichlet conditions $\bar\psi=\psi_p=\ln(N_v/N_A)$ and $\bar\psi=\psi_n=\ln(N_c/N_D)+E_c/U_T$
		are imposed at the p-contact ($z=0$) and the n-contact ($z=L$), respectively, and homogeneous
		Neumann (zero-flux) conditions are imposed on the lateral boundaries $r=0$ and $r=R$.
		
		Once $\bar\psi$ is known, the (dimensionless) carrier concentrations satisfy the linear convection--diffusion equations
		\begin{align}
			\label{eq:continuity_holes} %
			\nabla\!\cdot\!(\nabla\bar p + \bar p\,\nabla\bar\psi) &= 0, \qquad
			\bar p|_{z=0}=1,\;\; \bar p|_{z=L}=0, \\[4pt]
			\label{eq:continuity_elec} %
			\nabla\!\cdot\!(\nabla\bar n - \bar n\,\nabla\bar\psi) &= 0, \qquad
			\bar n|_{z=0}=0,\;\; \bar n|_{z=L}=N_D/N_A,
		\end{align}
		with homogeneous natural (zero-flux) conditions for both $\bar n$ and $\bar p$ on the lateral
		boundaries $r=0$ and $r=R$. These are exactly of the canonical form $-\nabla\!\cdot\!(\alpha\nabla u - \boldsymbol\beta\,u)=f$ with $\alpha=1$, $f=0$,
		$\boldsymbol\beta=-\nabla\bar\psi$ for holes and $\boldsymbol\beta=+\nabla\bar\psi$ for electrons.
		Physical concentrations are recovered as $p=\bar p\,C_0$, $n=\bar n\,C_0$, and the electrostatic potential as $\phi=\bar\psi\,U_T$.
		
		\newcommand{\sym}[1]{\makebox[3.8cm][l]{$#1$}}
		\begin{table}[h!]
			\centering
			\caption{Physical and material parameters for the GaAs p-n junction ($T=300\,\mathrm{K}$).}
			\label{tab:gaas_constants}
			\begin{tabular}{ll}
				\toprule
				Parameter & \hspace{-1cm}Value \\
				\midrule
				\sym{\varepsilon_s}     & \hspace{-1cm}$1.1422\times10^{-10}\,\mathrm{F/m}$  \\
				\sym{N_c}               & \hspace{-1cm}$4.3520\times10^{23}\,\mathrm{m^{-3}}$  \\
				\sym{N_v}               & \hspace{-1cm}$9.1396\times10^{24}\,\mathrm{m^{-3}}$  \\
				\sym{E_c}               & \hspace{-1cm}$1.424\,\mathrm{eV}$                    \\
				\sym{\mu_p}             & \hspace{-1cm}$0.04\,\mathrm{m^2/(V\!\cdot\!s)}$      \\
				\sym{N_A}               & \hspace{-1cm}$4.2042\times10^{24}\,\mathrm{m^{-3}}$ \\
				\sym{N_D}               & \hspace{-1cm}$4.3520\times10^{23}\,\mathrm{m^{-3}}$  \\
				\sym{q}                 & \hspace{-1cm}$1.602176634\times10^{-19}\,\mathrm{C}$  \\
				\sym{k_B}               & \hspace{-1cm}$1.380649\times10^{-23}\,\mathrm{J/K}$  \\
				\sym{T}                 & \hspace{-1cm}$300\,\mathrm{K}$                       \\
				\sym{U_T=k_BT/q}        & \hspace{-1cm}$2.5852\times10^{-2}\,\mathrm{V}$       \\
				\sym{\lambda^2}         & \hspace{-1cm}$4.3837\times10^{-6}$                 \\
				\bottomrule
			\end{tabular}
		\end{table}
		
		The physical parameters, listed in Table~\ref{tab:gaas_constants}, follow the GaAs
		benchmark of the numerical test in \cite{LeiW2024}; the doping densities and device length
		match their Section~6.3. With $U_T\approx 25.85\,\mathrm{mV}$, the contact potentials
		are $\psi_p=\ln(N_v/N_A)\approx0.777$ and $\psi_n=\ln(N_c/N_D)+E_c/U_T\approx55.083$,
		so $\bar\psi$ spans roughly $54.3$ Debye units ($\phi$ spans $1.404\,\mathrm{V}$).
		The Debye length is $\lambda_D=\sqrt{\varepsilon_sU_T/(q\,C_0)}\approx2.09\,\mathrm{nm}$,
		giving a Poisson scale $\lambda^2\approx4.38\times10^{-6}$.
		
		\paragraph{Discretization and mesh grading}
	Nonlinear Poisson--Boltzmann~\eqref{eq:poisson_pb} is solved by Newton iteration with
	$\lambda^2$-continuation ($\lambda^2=1\to\lambda^2_{\text{phys}}$); the exact Jacobian
	includes the diagonal Schur complement of the carrier densities $d\bar f/d\bar\psi=
	-(\bar p+\bar n)$.
	The $\lambda^2=1$ initial solve uses the doping-only linear problem
	$-\Delta\bar\psi = -\bar C$, which is a reliable starting guess for the continuation.
	We use the vertex-centered (box-dual) FVM Poisson solver of Section~\ref{subsec:second_order_2d}
	using Conjugate Gradient iteration together with the incomplete LU factorization as the preconditioner.
		
		Once the electrostatic field $\boldsymbol\beta=\mp\nabla\bar\psi$ is evaluated, the carrier
		continuity equations~\eqref{eq:continuity_holes}--\eqref{eq:continuity_elec}
		are solved on a bi-quadratic (Q2) space by the third-order complete-flux scheme
		CF3-RectC0 of Section~\ref{subsec:higher_order_rect} (Subsubsection~\ref{subsubsec:c0_quad}),
		i.e.\ by the full discrete operator $K_{V,1}+K_{V,2}=F_{V,1}+F_{V,2}$ of~\eqref{eq:CF3_scheme}.
		The Poisson problem is kept on the linear (Q1) space, so $\bar\psi$ is piecewise linear and
		the drift $\boldsymbol\beta$ is piecewise constant per cell.
		Here we note that
		the tangential correction $K_{V,2}$ is genuinely nonzero: for a bi-quadratic $u_h$ the
		tangential derivative $\partial u_h/\partial\eta$ varies along the dual face, so the diffusive
		part $\alpha\,\partial u_h/\partial\eta$ of the tangential flux $j_{\bm t}$ no longer cancels
		in the $\eta$-difference, in contrast to the linear case, where it cancels identically and
		$K_{V,2}$ reduces to the convective part alone (cf.~\eqref{eq,Kv2uh}).  Consequently the
		complete $K_{V,2}$, with both the diffusive and the convective tangential contributions, is
		engaged.  The source correction $F_{V,2}$ vanishes identically, since there is no
		recombination ($f\equiv0$ in~\eqref{eq:continuity_holes}--\eqref{eq:continuity_elec}).
		For reference, the second-order scheme of Section~\ref{subsec:second_order_2d} (linear
		carriers) yields essentially the same profiles and is used below as a consistency check.
		No Gummel iteration is performed: the Poisson and continuity solves are decoupled
		and run once each.
		
	The axial mesh is graded to resolve the depletion region near $z=L/2$.
	Cell widths follow a geometric progression: starting from $h_{\min}$ at the junction,
	they grow outward by a factor $\delta$ at each step until reaching a cap $h_{\max}$ in the bulk,
	then remain at $h_{\max}$ to the contacts.
	The 2D solver (quadrilateral mesh, $5$ cells in the radial direction)
	uses $h_{\min}=0.25\,\mathrm{nm}$, $\delta=1.05$, $h_{\max}=25\,\mathrm{nm}$,
	with the axial cell widths built exactly along the $z$-axis.
	All mesh parameters are summarized in Table~\ref{tab:mesh_grading}. Finally, we note that we used 
	the deal.II finite element library to implement the grid partition as well as the FVM solver \cite{dealii}.
	
	\begin{table}[h!]
		\centering
		\caption{Mesh grading parameters for the 2D p-n junction simulation.}
		\label{tab:mesh_grading}
		\begin{tabular}{lc}
			\toprule
			Parameter & 2D solver \\
			\midrule
			$h_{\min}$ (junction)  & $0.25\,\mathrm{nm}$  \\
			Growth factor $\delta$ & $1.05$               \\
			$h_{\max}$ (bulk cap)  & $25\,\mathrm{nm}$    \\
			Radial cells           & $5$                  \\
			\bottomrule
		\end{tabular}
	\end{table}
		
	\paragraph{Simulation results: linear versus bi-quadratic carriers}
		We compare two carrier discretizations on the same graded quadrilateral mesh
		(with the Poisson problem always on the linear Q1 space): the \emph{linear} (Q1)
		scheme CF2-Rect, whose normal edge flux is the Scharfetter--Gummel Bernoulli flux
		(cf.~Section~\ref{subsec:second_order_2d}), and the \emph{bi-quadratic} (Q2)
		CF3-RectC0 scheme, which engages the complete tangential correction $K_{V,2}$.
		Figure~\ref{fig:pn2d_q2_contours} shows the bi-quadratic contour fields of $\phi$,
		$\log_{10}p$, and $\log_{10}n$; the contours are vertical (independent of $r$),
		confirming the 1D-in-2D character of the solution.
		On the fine mesh of Table~\ref{tab:mesh_grading} the two discretizations agree
		closely and both reproduce Boltzmann equilibrium in the quasi-neutral regions.
		The higher-order space pays off once the depletion region is under-resolved.
		Figure~\ref{fig:pn2d_q1_vs_q2} overlays the axial profiles of $p$ and
		$n$ for the two schemes on the fine mesh and on a coarser mesh
		($h_{\min}=0.5\,\mathrm{nm}$, $\delta=1.15$).  On the fine mesh the curves coincide.
		On the coarser mesh the large cell-P\'eclet number in the depletion region drives
		the (non-monotone) linear solution to unphysical \emph{negative} concentrations
		($44$ hole and $1312$ electron nodes, out of $3280$), seen as the dashed curves
		dropping below zero, whereas the bi-quadratic solution remains entirely
		non-negative and stays close to Boltzmann equilibrium.  This illustrates the
		benefit of the higher-order space together with the complete tangential
		correction $K_{V,2}$. Thus, these results confirm that the complete-flux scheme can be applied to a
		semiconductor device problem in a fully decoupled fashion.
		
		\begin{figure}[htbp]
			\centering
			\includegraphics[width=0.95\linewidth]{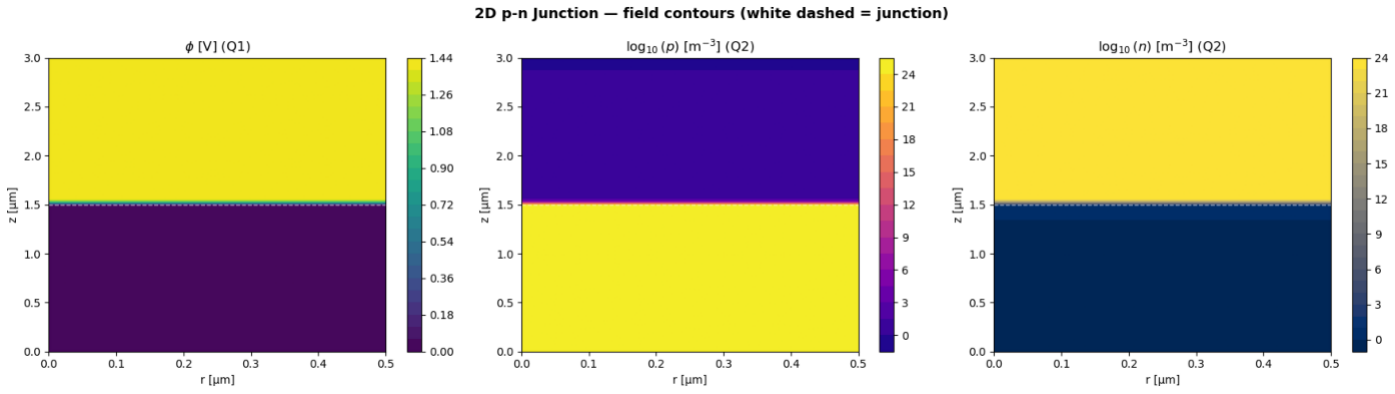}
			\caption{Two-dimensional p-n junction with bi-quadratic (Q2, CF3-RectC0) carriers:
				contour fields of $\phi$, $\log_{10}p$, and $\log_{10}n$. The vertical banding
				confirms $r$-independence. Junction at $z=1.5\,\mu\mathrm{m}$ (white dashed).}
			\label{fig:pn2d_q2_contours}
		\end{figure}
		
		\begin{figure}[htbp]
			\centering
			\includegraphics[width=0.95\linewidth]{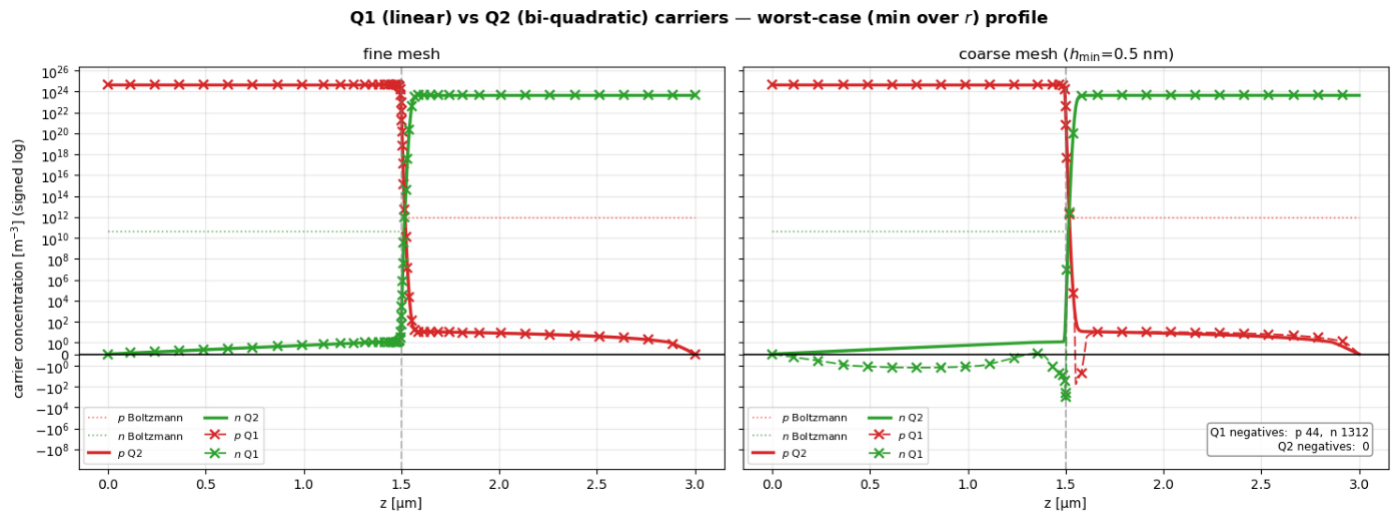}
			\caption{Two-dimensional p-n junction, linear (Q1, dashed with cross markers) versus
				bi-quadratic (Q2, solid) carriers: worst-case (minimum over $r$) axial profiles of
				$p$ and $n$ on a signed-log scale, on the fine mesh (left) and on a coarser mesh
				($h_{\min}=0.5\,\mathrm{nm}$, right), overlaid with Boltzmann equilibrium (dotted).
				On the coarse mesh the linear solution drops below zero in the depletion region,
				while the bi-quadratic solution stays non-negative.}
			\label{fig:pn2d_q1_vs_q2}
		\end{figure}
		

		\section{Conclusion}\label{sec:conclusion}
		
			We have developed a family of complete flux finite volume schemes for convection-diffusion equations on arbitrary grids. By deriving the exact normal flux from the underlying PDE via a Green's function, the framework yields discrete schemes that are automatically high-order accurate once the approximation space is chosen. The P\'{e}clet number entering the discrete flux is determined by the local extent of the control-volume boundary neighborhood, distinct from the usual mesh P\'{e}clet number and not fixed by the element size; this yields additional freedom in the construction. Concrete schemes include second-order discretizations on triangles and rectangles, third-order schemes on rectangles using $C^0$ Lagrange and $C^1$ B-spline spaces, and the extension to three-dimensional hexahedral and tetrahedral meshes. Numerical experiments in two and three dimensions confirm that the second-order schemes, whose control volumes are uniquely determined, attain the optimal convergence order, and that both third-order schemes attain optimal order and superconvergence, whether on a dual mesh not restricted to the Gauss--Legendre pattern (the $C^0$ Lagrange variant) or on the Greville dual mesh (the $C^1$ B-spline variant). In the two-dimensional semiconductor device simulations, the second- and third-order schemes are positivity-preserving. Immediate future work includes the study of stability and optimal error estimates for these schemes on unrestricted dual meshes and for arbitrary P\'{e}clet numbers, as well as the analysis of their positivity preservation.

		\section*{Acknowledgments}
WL is partially supported by the National Natural Science Foundation of China
(Grant No. 12301496 and 12431015) and the Fundamental Research Fund for the Central Universities, University of Electronic Science and Technology of China (Grant No. Y030242063002065). 
LX is partially supported by the National Natural Science Foundation of China
(Grant No. 62231016 and 12431015) and the Foundation for Innovative Research Groups of the Natural Science Foundation of Sichuan Province (No. 2025NSFTD0004). PY is partially supported by the National Natural Science Foundation of China (Grant No. 12501537).
	
	\bibliographystyle{siamplain}
	\bibliography{CFFVrefer}

\begin{thebibliography}{10}

\bibitem{Aavatsmark2002}
{\sc I.~Aavatsmark}, {\em An introduction to multipoint flux approximations for
  quadrilateral grids}, Comput. Geosci., 6 (2002), pp.~405--432.

\bibitem{dealii}
{\sc D.~Arndt, W.~Bangerth, M.~Bergbauer, B.~Blais, M.~Fehling,
  R.~Gassm\"{o}ller, T.~Heister, L.~Heltai, M.~Kronbichler, M.~Maier, P.~Munch,
  S.~Scheuerman, B.~Turcksin, S.~Uzunbajakau, D.~Wells, and M.~Wichrowski},
  {\em The deal.ii library, version 9.7}, Journal of Numerical Mathematics, 33
  (2025), pp.~403--415.

\bibitem{BankRose1987}
{\sc R.~E. Bank and D.~J. Rose}, {\em Some error estimates for the box method},
  SIAM J. Numer. Anal., 24 (1987), pp.~777--787.

\bibitem{BrooksHughes1982}
{\sc A.~N. Brooks and T.~J.~R. Hughes}, {\em Streamline upwind/petrov-galerkin
  formulations for convection dominated flows with particular emphasis on the
  incompressible navier-stokes equations}, Comput. Methods Appl. Mech. Engrg.,
  32 (1982), pp.~199--259.

\bibitem{Cai1991}
{\sc Z.~Cai}, {\em On the finite volume element method}, Numer. Math., 58
  (1991), pp.~713--735.

\bibitem{ChenZ2012}
{\sc Z.~Chen, J.~Wu, and Y.~Xu}, {\em Higher-order finite volume methods for
  elliptic boundary value problems}, Adv. Comput. Math., 37 (2012),
  pp.~191--253.

\bibitem{Cheng2021}
{\sc H.~M. Cheng and J.~ten Thije~Boonkkamp}, {\em A generalised complete flux
  scheme for anisotropic advection-diffusion equations}, Adv. Comput. Math., 47
  (2021), pp.~Paper No. 19, 26.

\bibitem{CoddingtonLevinson1955}
{\sc E.~A. Coddington and N.~Levinson}, {\em Theory of Ordinary Differential
  Equations}, McGraw-Hill, New York, 1955.

\bibitem{Eymard2000}
{\sc R.~Eymard, T.~Gallou\"et, and R.~Herbin}, {\em Finite volume methods}, in
  Handbook of Numerical Analysis, North-Holland, Amsterdam, 2000,
  pp.~713--1020.

\bibitem{Farrell2017}
{\sc P.~Farrell and A.~Linke}, {\em Uniform second order convergence of a
  complete flux scheme on unstructured 1{D} grids for a singularly perturbed
  advection-diffusion equation and some multidimensional extensions}, J. Sci.
  Comput., 72 (2017), pp.~373--395.

\bibitem{Forsyth1991}
{\sc P.~A. Forsyth}, {\em A control volume finite element approach to {NAPL}
  groundwater contamination}, SIAM J. Sci. Statist. Comput., 12 (1991),
  pp.~1029--1057.

\bibitem{Hermeline2000}
{\sc F.~Hermeline}, {\em A finite volume method for the approximation of
  diffusion operators on distorted meshes}, J. Comput. Phys., 160 (2000),
  pp.~481--499.

\bibitem{HughesCottrellBazilevs2005}
{\sc T.~J.~R. Hughes, J.~A. Cottrell, and Y.~Bazilevs}, {\em Isogeometric
  analysis: {CAD}, finite elements, {NURBS}, exact geometry and mesh
  refinement}, Comput. Methods Appl. Mech. Engrg., 194 (2005), pp.~4135--4195.

\bibitem{KamberG2022}
{\sc G.~Kamber, H.~Gotovac, V.~Kozuli\'c, and B.~Gotovac}, {\em 2-{D} local hp
  adaptive isogeometric analysis based on hierarchical {F}up basis functions},
  Comput. Methods Appl. Mech. Engrg., 398 (2022), pp.~Paper No. 115272, 32.

\bibitem{KamberG2020}
{\sc G.~Kamber, H.~Gotovac, V.~Kozuli\'c, L.~Malenica, and B.~z. Gotovac}, {\em
  Adaptive numerical modeling using the hierarchical {F}up basis functions and
  control volume isogeometric analysis}, Internat. J. Numer. Methods Fluids, 92
  (2020), pp.~1437--1461.

\bibitem{LeiW2024}
{\sc W.~Lei, S.~Piani, P.~Farrell, N.~Rotundo, and L.~Heltai}, {\em A weighted
  hybridizable discontinuous {G}alerkin method for drift-diffusion problems},
  J. Sci. Comput., 99 (2024), pp.~Paper No. 33, 26.

\bibitem{LeVeque2002}
{\sc R.~J. LeVeque}, {\em Finite Volume Methods for Hyperbolic Problems},
  Cambridge University Press, 2002.

\bibitem{LiR2000}
{\sc R.~Li, Z.~Chen, and W.~Wu}, {\em Generalized difference methods for
  differential equations}, vol.~226, Marcel Dekker, Inc., New York, 2000.

\bibitem{LiYangZhangPPR}
{\sc Y.~Li, P.~Yang, and Z.~Zhang}, {\em Polynomial preserving recovery for the
  finite volume element methods under simplex meshes}, Math. Comp., 94 (2025),
  pp.~611--645.

\bibitem{LinY2015}
{\sc Y.~Lin, M.~Yang, and Q.~Zou}, {\em {$L^2$} error estimates for a class of
  any order finite volume schemes over quadrilateral meshes}, SIAM J. Numer.
  Anal., 53 (2015), pp.~2009--2029.

\bibitem{Liu2013}
{\sc L.~Liu, J.~van Dijk, J.~H.~M. ten Thije~Boonkkamp, D.~B. Mihailova, and
  J.~J. A.~M. van~der Mullen}, {\em The complete flux scheme---error analysis
  and application to plasma simulation}, J. Comput. Appl. Math., 250 (2013),
  pp.~229--243.

\bibitem{lv2012optimal}
{\sc J.~Lv and Y.~Li}, {\em Optimal biquadratic finite volume element methods
  on quadrilateral meshes}, SIAM J. Numer. Anal., 50 (2012), pp.~2379--2399.

\bibitem{Patankar1980}
{\sc S.~Patankar}, {\em Numerical Heat Transfer and Fluid Flow}, Hemisphere
  Publishing Corporation, Washington, DC, 1980.

\bibitem{RoosStynesTobiska2008}
{\sc H.~G. Roos, M.~Stynes, and L.~Tobiska}, {\em Robust numerical methods for
  singularly perturbed differential equations: convection-diffusion-reaction
  and flow problems}, Springer, 2nd~ed., 2008.

\bibitem{Scharfetter1969}
{\sc D.~L. Scharfetter and H.~K. Gummel}, {\em Large-signal analysis of a
  silicon read diode oscillator}, IEEE Transactions on Electron Devices, 16
  (1969), pp.~64--77.

\bibitem{SeinfeldPandis2016}
{\sc J.~H. Seinfeld and S.~N. Pandis}, {\em Atmospheric Chemistry and Physics:
  From Air Pollution to Climate Change}, Wiley, 3rd~ed., 2016.

\bibitem{Selberherr1984}
{\sc S.~Selberherr}, {\em Analysis and Simulation of Semiconductor Devices},
  Springer-Verlag, 1984.

\bibitem{ten2005complete}
{\sc J.~H.~M. ten Thije~Boonkkamp}, {\em A complete flux scheme for
  one-dimensional combustion simulation}, Hermes Science Publishing, London,
  2005, pp.~573--583.

\bibitem{ten2011}
{\sc J.~H.~M. ten Thije~Boonkkamp and M.~J.~H. Anthonissen}, {\em The finite
  volume-complete flux scheme for advection-diffusion-reaction equations}, J.
  Sci. Comput., 46 (2011), pp.~47--70.

\bibitem{Thiart1990}
{\sc G.~D. Thiart}, {\em Improved finite-difference scheme for the solution of
  convection-diffusion problems with the simplen algorithm}, Numer. Heat
  Transf. Part B, 18 (1990), pp.~81--95.

\bibitem{Wang2016}
{\sc X.~Wang and Y.~Li}, {\em $l^2$ error estimates for high order finite
  volume methods on triangular meshes}, SIAM J. Numer. Anal., 54 (2016),
  pp.~2729--2749.

\bibitem{XuZou2009}
{\sc J.~Xu and Q.~Zou}, {\em Analysis of linear and quadratic simplicial finite
  volume methods for elliptic equations}, Numer. Math., 111 (2009),
  pp.~469--492.

\bibitem{ZhangZ2015}
{\sc Z.~Zhang and Q.~Zou}, {\em Vertex-centered finite volume schemes of any
  order over quadrilateral meshes for elliptic boundary value problems}, Numer.
  Math., 130 (2015), pp.~363--393.

\bibitem{ZhengBennett2002}
{\sc C.~Zheng and G.~D. Bennett}, {\em Applied Contaminant Transport Modeling},
  Wiley, 2nd~ed., 2002.

\end{thebibliography}
\end{document}